\documentclass[letterpaper, 10 pt, conference]{ieeeconf}

\IEEEoverridecommandlockouts
\usepackage{graphics}
\graphicspath{{./Figures/}}
\usepackage{epsfig}
\usepackage{times}
\usepackage{amsmath}
\usepackage{amssymb}
\usepackage{mathrsfs} % for \mathscr
\usepackage{subcaption}
\usepackage{wrapfig}
\usepackage{comment}
\usepackage{booktabs}
\usepackage{multirow}
\usepackage{mathtools} % for \coloneqq, etc.
\usepackage{bm}        % for bold math
\usepackage{xcolor}

\usepackage{algorithm}
\usepackage[indLines]{algpseudocodex}
\algrenewcommand\algorithmiccomment[2][0em]{%
  \hspace*{#1}\textcolor{gray}{\footnotesize$\triangleright$~\textit{#2}}%
}

\algtext*{EndIf}
\algtext*{EndWhile}

\renewcommand{\d}{\mathrm{d}}
\newcommand{\E}{\mathbb{E}}
\renewcommand{\P}{\mathbb{P}}
\newcommand{\R}{\mathbb{R}}

\renewcommand{\S}{\mathbb{S}}

\newcommand{\scriptN}{\mathcal{N}}
\newcommand{\scriptX}{\mathcal{X}}
\newcommand{\scriptU}{\mathcal{U}}
\newcommand{\scriptJ}{\mathcal{J}}
\newcommand{\scriptF}{\mathcal{F}}
\newcommand{\scriptP}{\mathcal{P}}

\newcommand{\scriptD}{\mathcal{D}}
\newcommand{\scriptC}{\mathcal{C}}
\newcommand{\scriptQ}{\mathcal{Q}}
\newcommand{\scriptL}{\mathcal{L}}

\newcommand{\vct}[1]{\boldsymbol{#1}}

\newtheorem{thm}{Theorem}
\newtheorem{assum}{Assumption}

\newtheorem{problem}{Problem}

\newtheorem{rem}{Remark}

\newtheorem{lem}{Lemma}

\title{\LARGE \bf
Free Final Time Adaptive Mesh Covariance Steering \\ via Sequential Convex Programming
}

\author{Joshua Pilipovsky%
\thanks{J. Pilipovsky is a Senior Research Engineer at RTX Technology Research Center (RTRC), East Hartford, CT, 06119, USA. Email: joshua.y.pilipovsky@rtx.com}
}

\begin{document}

\maketitle
\thispagestyle{empty}
\pagestyle{empty}

%%%%%%%%%%%%%%%%%%%%%%%%%%%%%%%%%%%%%%%%%%%%%%%%%%%%%%%%%%%%%%%%%%%%%%%%%%%%%%%%
\begin{abstract}
In this paper we develop a sequential-convex-programming (SCP) framework for \emph{free-final-time} covariance steering of nonlinear stochastic differential equations (SDEs) subject to both additive and multiplicative diffusion.
We cast the free-final-time objective through a time-normalization and introduce per-interval time-dilation variables that induce an \emph{adaptive} discretization mesh, enabling the simultaneous optimization of the control policy and the temporal grid.
A central difficulty is that, under multiplicative noise, accurate covariance propagation within SCP requires retaining the first-order diffusion linearization and its coupling with time dilation.
We therefore derive the exact local linear stochastic model (preserving the multiplicative structure) and introduce a tractable discretization that maintains the associated diffusion terms, after which each SCP subproblem is solved via conic/semidefinite covariance-steering relaxations with terminal moment constraints and state/control chance constraints.
Numerical experiments on a nonlinear double-integrator with drag and velocity-dependent diffusion validate free-final-time minimization through adaptive time allocation and improved covariance accuracy relative to frozen-diffusion linearizations.
\end{abstract}

%%%%%%%%%%%%%%%%%%%%%%%%%%%%%%%%%%%%%%%%%%%%%%%%%%%%%%%%%%%%%%%%%%%%%%%%%%%%%%%%
\section{INTRODUCTION}
Successive convexification (SCvx) and, more broadly, sequential convex programming (SCP) methods have become a leading paradigm for nonconvex optimal control due to their polynomial-time convex subproblems, strong empirical reliability, and increasingly mature convergence/feasibility guarantees \cite{Mao2018SCvx,Oguri2023ALSCvx,Elango2025ctSCvx}. In aerospace guidance and trajectory optimization, these methods have enabled real-time-capable pipelines for problems with nonlinear dynamics, nonconvex state/control constraints, and free-final-time objectives \cite{Szmuk2018FreeFinalTime}. In their standard form, SCvx discretizes the continuous-time problem on a fixed temporal mesh, convexifies the dynamics and nonconvex constraints about a reference trajectory, and enforces progress through trust-region and penalty mechanisms (e.g., virtual controls and constraint softening).

Extending this paradigm to \emph{stochastic} optimal control is substantially more challenging because one must simultaneously steer the mean trajectory and propagate/control uncertainty under probabilistic safety requirements. Covariance control (and, in finite horizon form, covariance steering (CS)) addresses this by shaping the first two moments of the state distribution via feedback control, and has led to convex, scalable synthesis methods for linear stochastic systems with terminal covariance and chance constraints \cite{Bakolas2018dtCS,Chen2016ctCS,Okamoto2018CCS,Liu2025dtCS}. These developments have supported a growing body of constrained stochastic guidance applications, including spacecraft proximity operations and low-thrust trajectory design \cite{Pilipovsky2021IRA,Kumagai2025Cislunar,Pilipovsky2020Interplanetary}.

A natural next step is to combine the tractability and reliability of SCvx with the distribution-shaping guarantees of CS for \emph{nonlinear} stochastic systems. Iterative covariance steering (iCS) \cite{Kumagai2025Cislunar,Ridderhof2019iCS,Oguri2022iCS,Benedikter2022iCS} follows this strategy by repeatedly linearizing a nonlinear stochastic differential equation (SDE), solving a convex CS subproblem for the local model, and updating the reference within an SCP loop. For additive noise, local covariance propagation is relatively straightforward; however, many applications are more naturally modeled with \emph{multiplicative} noise, where the diffusion depends on the state and/or control (e.g., atmospheric-density-driven entry uncertainty, attitude-dependent disturbances, and throttle-/pointing-dependent propulsion uncertainty). In this setting, covariance dynamics depend on the first-order diffusion linearization, and the resulting CS constraints can become nonconvex unless suitable liftings/relaxations are introduced \cite{Balci2022MixedMN,Knaup2023ParametricCS}. Moreover, within SCP/iCS, mismatch between the covariance propagation used in the convex subproblem and the true local stochastic dynamics can accumulate and degrade constraint satisfaction and convergence.

Many guidance and mission-design problems are also intrinsically \emph{free-final-time}. In deterministic SCvx, this is commonly handled by time normalization and a time-dilation variable, which preserves a fixed number of discretization nodes while introducing a convex surrogate for final time \cite{Szmuk2018FreeFinalTime}. Adaptive-mesh SCP variants further generalize this idea by optimizing interval-wise time dilation or reallocating discretization density \cite{KumagaiOguri2024AdaptiveMesh,Elango2025ctSCvx,ZHOU2021AdaptiveMesh,Tafazzol2025AdaptiveMesh}. In stochastic problems, however, time dilation scales not only the drift but also the diffusion, and therefore directly affects covariance growth. To our knowledge, a principled CS formulation that jointly optimizes the temporal mesh and the covariance steering policy for nonlinear SDEs with multiplicative noise has not been previously developed. The purpose of this paper is to close this gap.

We consider a nonlinear continuous-time SDE with control-affine (possibly nonlinear) drift and state/control-dependent diffusion, subject to terminal mean/covariance requirements and convex chance constraints along the trajectory, and seek a feedback policy that \emph{minimizes final time}. We time-normalize the dynamics to $\tau\in[0,1]$ and introduce per-interval time-dilation variables $\{\sigma_k\}_{k=0}^{N-1}$, which induce an adaptive physical-time mesh while keeping the normalized discretization fixed. A key technical ingredient is a first-order discretization of the local linear multiplicative-noise model that preserves the diffusion linearization \emph{including its time-dilation dependence}, together with one-step mean-square error bounds. This yields more accurate covariance propagation inside the SCP loop under multiplicative noise and enables the integration of \emph{lossless} convex relaxations for multiplicative-noise CS within each convex subproblem.

\subsection*{Contributions}
The main contributions of this work are as follows.
\begin{enumerate}
    \item To the authors' knowledge, the first free-final-time covariance-steering formulation for nonlinear SDEs, using interval-wise time-dilation variables to jointly optimize the control policy and temporal mesh within an SCP/iCS framework.
    \item An exact first-order diffusion linearization under time normalization (including time-dilation coupling), together with a discrete-time local stochastic model and mean-square error bounds that support accurate covariance propagation under multiplicative noise.
    \item A tractable SCvx$^\star$-based solution method whose convex subproblems incorporate lossless convex relaxations for multiplicative-noise covariance steering and convex chance constraints, demonstrated on a nonlinear double-integrator example with drag and velocity-dependent diffusion.
\end{enumerate}

\section{NOTATION}\label{sec:notation}
All random objects are defined on a common probability space $(\Omega,\scriptF,\P)$.
Deterministic vectors are denoted by lowercase letters ($u\in\R^{m}$), matrices by uppercase letters ($V\in\R^{n\times m}$), and random vectors by boldface ($\vct x\in\R^{n}$). We denote the sets of symmetric, positive semidefinite, and positive definite matrices by $\S^n$, $\S^n_{+}$, and $\S^n_{++}$, respectively.
A continuous-time signal over $t\in[t_0,t_1]$ is written $\{x_t\}_{t\in[t_0,t_1]}$. Given a filtration $\{\scriptF_\tau\}_{\tau\ge 0}$, we write $\E_k[\cdot]\coloneqq \E[\cdot\mid \scriptF_{\tau_k}]$.
For $\zeta\in\R^r$, $[\zeta]_+ \coloneqq [\max(\zeta_1,0),\ldots,\max(\zeta_r,0)]^\intercal$ represents the vectorized maximum.
For non-negative quantities $a$ and $b$, we write $a\lesssim b$ if there exists a finite constant $C>0$ such that $a\le Cb.$
Finally, $[a;b]$ and $[A;B]$ denote vertical concatenation of compatible vectors and matrices.

%%%%%%%%%%%%%%%%%%%%%%%%%%%%%%%%%%%%%%%%%%%%%%%%%%%%%%%%%%%%%%%%%%%%%%%%%%%%%%%%
\section{PROBLEM STATEMENT}\label{sec:problem_statement}

Consider a dynamical system governed by the nonlinear stochastic differential equation (SDE)
\begin{equation}\label{eq:nl-sde}
    \d\vct x_t = f\big(\vct x_t, \vct u_t\big)\, \d t + g\big(\vct x_t, \vct u_t\big)\, \d \vct w_t, \quad t \in [0,t_f],
\end{equation}
where $\vct x_t \in \R^{n}$ is the state, $\vct u_t \in \R^{m}$ is the control input, $\vct w_t\in\R^{d}$ is a $d$-dimensional Brownian motion, and $t_f > 0$ is the final time.
In this work, we assume that $t_f$ is \emph{free}.
The state is initially normally distributed as
\begin{equation}\label{eq:init-cond}
    \vct x_0 \sim \scriptN(\mu_i, \Sigma_i),
\end{equation}
where $\mu_i \in \R^{n}$ and $\Sigma_i \in \S_{++}^{n}$, and the terminal distribution is constrained to be
\begin{equation}\label{eq:term-cond}
    \vct x_{t_f} \sim \scriptN(\mu_f, \Sigma_f),
\end{equation}
where $\mu_f \in \R^{n}$ and $\Sigma_f \in \S_{++}^{n}$.
Let $\scriptX \subset \R^{n}$ and $\scriptU \subset \R^{m}$ denote the allowable convex sets for state and input, which we assume are given by the polytopes
\begin{equation}~\label{eq:half-spaces}
    \scriptX = \bigcap_{j=1}^{N_x} \{x : \alpha_j^\intercal x + \beta_j \leq 0\},\ \scriptU = \bigcap_{j=1}^{N_u} \{u : a_j^\intercal u + b_j \leq 0\},
\end{equation}
where $\alpha_j\in\R^n,a_j\in\R^m$ and $\beta_j,b_j \in \R$ define the individual half-spaces.
The above choice is for simplicity; other convex (e.g., norm-based) or even nonconvex (e.g., obstacle-avoidance) admissible regions can be handled within the SCP framework in Section~\ref{sec:scp}~\cite{Oguri2022iCS,Lew2020ccSCP}.
We impose continuous-time joint chance constraints
\begin{subequations}\label{eq:ct-cc}
    \begin{align}
        \P(\vct x_t \in \scriptX \ \forall t \in [0,t_f]) &\geq 1 - \Delta_x, \\
        \P(\vct u_t \in \scriptU \ \forall t \in [0,t_f]) &\geq 1 - \Delta_u,
    \end{align}
\end{subequations}
where $\Delta_x, \Delta_u \in (0, 0.5]$ are joint risk tolerances.
We consider an objective function which aims to minimize the regularized final time
\begin{equation}\label{eq:ct-obj}
    \scriptJ(\vct x, \vct u, t_f) = \eta t_f + \scriptJ_{\mathrm{reg}}(\vct x, \vct u),
\end{equation}
where $\eta > 0$ weighs the free-final time, and $\scriptJ_{\mathrm{reg}}\geq 0$ denotes a mean/covariance regularization which will be precisely defined in Section~\ref{subsec:convex-reduction}.
We define the set $\pi$ of \textit{admissible} control inputs as the set of (random) control signals $\{\vct u_t\}_{t\in[0,t_f]}$ where the input $\vct u_t$ is an affine function of the state.
In total, the \textbf{F}ree-\text{F}inal \textbf{T}ime \textbf{C}ovariance \textbf{S}teering (FFT-CS) is given as follows
\begin{problem}[FFT-CS]\label{prob:FFT-CS}
    For a given initial distribution \eqref{eq:init-cond}, find a control policy that minimizes \eqref{eq:ct-obj} subject to \eqref{eq:nl-sde} and \eqref{eq:ct-cc}, such that \eqref{eq:term-cond} holds.
\end{problem}
%

%%%%%%%%%%%%%%%%%%%%%%%%%%%%%%%%%%%%%%%%%%%%%%%%%%%%%%%%%%%%%%%%%%%%%%%%%%%%%%%%
\section{PROBLEM REFORMULATION}\label{sec:reformulation}~\label{sec:prob-reform}
To find a locally optimal solution to Problem~\ref{prob:FFT-CS}, we utilize the framework of Sequential Convex Programming (SCP), which aims to iteratively build and solve approximate (yet tractable) convex sub-problems of Problem~\ref{prob:FFT-CS}.
The overall iterative scheme is presented in Algorithm~\ref{alg:scvxstar} and is discussed in detail in Section~\ref{sec:scp}.
In this section, we outline the main steps that must be performed to arrive at a tractable convex CS sub-problem, namely (i) time scaling, (ii) linearization, (iii) discretization, and (iv) convex reduction.

\subsection{Time Scaling}
Since the final time $t_f$ is free, it is convenient (and in this context, necessary) to scale the time horizon $[0, t_f]$ to a fixed interval.
Let $t:[0,1]\rightarrow [0,t_f]$ be a continuously differentiable and nondecreasing function of a parameter $\tau$ with $t(0)=0$ and $t(1)=t_f$, and define the \emph{time-dilation} factor $\sigma_\tau \coloneqq \d t/\d\tau \ge 0$.
Under this change of variables, the dynamics \eqref{eq:nl-sde} are equivalently written on the \emph{fixed} interval $\tau\in[0,1]$ as
\begin{equation}\label{eq:nl-sde-scaled}
    \d\vct x_{\tau} = \sigma_\tau f(\vct x_\tau, \vct u_\tau)\,\d\tau + \sqrt{\sigma_\tau}\, g(\vct x_\tau, \vct u_\tau)\,\d\vct w_{\tau}.
\end{equation}
Moreover, it is straightforward to see that
\begin{equation}\label{eq:tf-sigma}
    t_f=\int_{0}^{1} \sigma_\tau\,\d\tau,
\end{equation}
so the objective \eqref{eq:ct-obj} is equivalently written as
\begin{equation}\label{eq:obj-time-dilation}
    \scriptJ(\vct x, \vct u, \sigma) = \eta \int_{0}^{1} \sigma_\tau\, \d\tau + \scriptJ_{\mathrm{reg}}(\vct x, \vct u).
\end{equation}
The nonlinear SDE \eqref{eq:nl-sde-scaled} and objective \eqref{eq:obj-time-dilation} are now written over a fixed horizon, at the cost of introducing the additional decision variable $\sigma_\tau$.
This time scaling procedure, which is now customary in deterministic optimal control SCP algorithms~\cite{Szmuk2018FreeFinalTime,Elango2025ctSCvx,Kamath2023seco}, has scarcely been applied to stochastic optimal control problems.

Throughout, the normalized state process $\{\vct x_\tau\}_{\tau\in[0,1]}$ is assumed adapted to the natural filtration
\begin{equation}~\label{eq:nat-filt}
    \mathcal F_\tau \;\coloneqq\; \sigma\!\big(\vct x_0,\{\vct w_s:0\le s\le \tau\}\big),\qquad \tau\in[0,1],
\end{equation}
generated by the initial condition $\vct x_0$ and the driving Brownian motion $\d\vct w_{\tau}$.
For convenience, explicit expressions for the linearization and discretization matrices that appear in the sequel are collected in Appendix~\ref{app:ct-lin} and Appendix~\ref{app:exact-disc}.

\subsection{Linearization}~\label{subsec:lin}
We linearize \eqref{eq:nl-sde-scaled} about a reference trajectory $\hat z_{\tau} \coloneqq (\hat x_{\tau}, \hat u_{\tau}, \hat\sigma_{\tau})$.
For clarity, write the diffusion column-wise as $g=[g_1,\dots,g_d]\in\R^{n\times d}$.
%
%\begin{align}\label{eq:linear-sde-scaled}
%    &\d\vct x_{\tau} = \left(A_{\tau} \vct x_{\tau} + B_{\tau} \vct u_{\tau} + c_{\tau} \sigma_{\tau} + d_{\tau}\right)\, \d\tau \nonumber \\
%    &\quad+ \sum_{i=1}^{d} \left(\tilde{A}_{\tau}^{(i)} \vct x_{\tau} + \tilde{B}_{\tau}^{(i)} \vct u_{\tau} + \tilde{c}_{\tau}^{(i)} \sigma_{\tau} + \tilde{d}_{\tau}^{(i)}\right)\, \d\vct w_{\tau}^{(i)}.
%\end{align}
Define the augmented control input $\tilde{\vct u}_{\tau} \coloneqq [\vct u_{\tau}^\intercal, \sigma_{\tau}]^\intercal\in\R^{m+1}$ and the matrices
\begin{equation*}
    F_\tau \coloneqq \begin{bmatrix}B_\tau & c_\tau\end{bmatrix},\qquad
    \tilde F_\tau^{(i)} \coloneqq \begin{bmatrix}\tilde B_\tau^{(i)} & \tilde c_\tau^{(i)}\end{bmatrix}.
\end{equation*}
Then the first-order It\^{o} linearization yields
\begin{align}\label{eq:linear-sde-aug}
    \d\vct x_{\tau} &= (A_\tau \vct x_\tau + F_\tau \tilde{\vct u}_\tau + d_\tau)\, \d\tau \nonumber\\
    &\hspace{1cm}+ \sum_{i=1}^{d} \left(\tilde{A}_{\tau}^{(i)} \vct x_\tau + \tilde{F}_\tau^{(i)} \tilde{\vct u}_\tau + \tilde{d}_{\tau}^{(i)}\right)\, \d\vct w_\tau^{(i)}.
\end{align}
The matrices and vectors involved in this linearization are given explicitly in Appendix~\ref{app:ct-lin}.

\begin{rem}\label{rem:exact-diff-lin}
In contrast to prior SCP/iCS formulations~\cite{Oguri2022iCS,Benedikter2022iCS,Kumagai2025Cislunar,Ridderhof2019iCS} that \emph{freeze} the diffusion at the reference (i.e., $\tilde A_\tau^{(i)}=\tilde F_\tau^{(i)}=0$), \eqref{eq:linear-sde-aug} retains the full first-order diffusion linearization, including the time-dilation dependence through $\sqrt{\sigma_\tau}$.
This yields more accurate local moment propagation inside SCP, at the cost of a multiplicative-noise linear model (and hence generally non-Gaussian state distributions).
\end{rem}

\subsection{Discretization}
We discretize \eqref{eq:linear-sde-aug} on a partition $\scriptP \coloneqq (\tau_0, \ldots, \tau_{N})$ with $\tau_0 = 0$, $\tau_N = 1$ and step sizes $\Delta\tau_k \coloneqq \tau_{k+1}-\tau_k$.
It is important to note that even if the time-scaled step size is fixed, i.e., $\Delta \tau_k \equiv \Delta\tau$, the physical time step need not be constant, since $\Delta t_k = \int_{0}^{\Delta\tau} \sigma_\tau\, \d\tau$.
For brevity, we write $\vct x_{t_k} = \vct x_k,\, \vct u_{t_k} = \vct u_k$, and so on for all other decision variables.
We assume a zero-order hold (ZOH) on the augmented control input, that is,
\begin{equation}\label{eq:zoh}
    \tilde{\vct u}_\tau \equiv \tilde{\vct u}_{k}, \qquad \forall \tau \in [\tau_k, \tau_{k+1}).
\end{equation}
Under ZOH on $\sigma_\tau$, the objective \eqref{eq:obj-time-dilation} is approximated by the Riemann sum
\begin{equation}\label{eq:obj-riemann}
    \scriptJ(\vct x_k, \tilde{\vct u}_k) \approx \eta \sum_{k=0}^{N-1} \sigma_k \Delta \tau_k + \scriptJ_{\mathrm{reg}}(\vct x_k, \vct u_k),
\end{equation}
which is linear in the decision variables $\{\sigma_k\}$.
Let $\Phi(\tau,s)$ denote the state transition matrix (STM) associated with $\dot z = A_\tau z$, i.e.,
\begin{equation}\label{eq:stm-def}
    \frac{\partial}{\partial \tau}\Phi(\tau,s)=A_\tau\Phi(\tau,s), \qquad \Phi(s,s)=I.
\end{equation}
Under the ZOH assumption \eqref{eq:zoh}, the linear SDE \eqref{eq:linear-sde-aug} admits the \textit{exact} mild solution on $[\tau_k,\tau_{k+1}]$:
\begin{align}
    &\vct x_{k+1}
    = A_k \vct x_k + \vct{\nu}_k \nonumber\\
    &\quad+ \sum_{i=1}^{d} \int_{\tau_k}^{\tau_{k+1}}
    \Phi(\tau_{k+1},\tau)\Big(\tilde{A}_\tau^{(i)}\vct x_\tau + \tilde{\vct{\nu}}_\tau^{(i)}\Big)\,\d\vct w_\tau^{(i)},~\label{eq:exact-discrete}
\end{align}
where $\vct\nu_k \coloneq F_k \tilde{\vct u}_k + d_k,\, \tilde{\vct\nu}_\tau^{(i)} \coloneq \tilde{F}_\tau^{(i)} \tilde{\vct u}_k + \tilde{d}_\tau^{(i)}$. The deterministic discretization terms $(A_k,F_k,d_k)$ together with the averaged diffusion coefficients $(\tilde A_k^{(i)},\tilde F_k^{(i)},\tilde d_k^{(i)})$ are given in Appendix~\ref{app:exact-disc}.
Therein, we also present a computationally efficient (and parallelizable) procedure to compute these matrices by solving a small augmented linear ODE, avoiding per-interval numerical quadrature.
The only intractable term in \eqref{eq:exact-discrete} is the stochastic integral, since $\vct x_\tau$ appears inside the integrand.
In the following, we overcome this issue by (i) freezing the state inside the integrand and (ii) approximating with a constant integrand.

\begin{assum}[Frozen state in the diffusion integrand]\label{assum:freeze-diff}
On each interval $[\tau_k,\tau_{k+1}]$, we approximate the diffusion integrand in \eqref{eq:exact-discrete} by freezing the state at the left endpoint:
\begin{equation}\label{eq:freeze-x}
    \vct x_\tau \approx \vct x_k,\qquad \forall \tau\in[\tau_k,\tau_{k+1}].
\end{equation}
\end{assum}

For each channel $i\in\{1,\dots,d\}$ and $\tau \in [\tau_k, \tau_{k+1}]$, define the $\scriptF_{\tau_k}$-measurable integrand
\begin{equation}\label{eq:def-H}
    H_{k,i}(\tau;\vct x_k,\tilde{\vct u}_k)
    \coloneqq \Phi(\tau_{k+1},\tau)\Big(\tilde{A}^{(i)}_\tau \vct x_k + \tilde{\vct{\nu}}^{(i)}_\tau\Big)\in\R^{n}.
\end{equation}
Stack the channels as
\begin{align}
    \hspace{-0.3cm}H_k(\tau;\vct x_k,\tilde{\vct u}_k) &\coloneqq \begin{bmatrix} H_{k,1}(\tau) & \cdots & H_{k,d}(\tau)\end{bmatrix}\in \R^{n\times d},\label{eq:Hstack}\\
    \Delta\vct w_k &\coloneqq \begin{bmatrix}\Delta\vct w_{k,1} & \cdots & \Delta\vct w_{k,d}\end{bmatrix}^\intercal\in\R^{d},\label{eq:dwstack}
\end{align}
where $\Delta\vct w_{k,i}\coloneqq \vct w^{(i)}_{\tau_{k+1}}-\vct w^{(i)}_{\tau_k}$ and hence $\Delta\vct w_k \sim \scriptN(0,\Delta\tau_k I_d)$ with statistics
\begin{equation}\label{eq:dw-moments}
    \E[\Delta\vct w_k]=0,\qquad \E[\Delta\vct w_k\Delta\vct w_k^\intercal]=\Delta\tau_k I_d.
\end{equation}
Moreover, since $\vct x_k$ is $\scriptF_{\tau_k}$-measurable and Brownian increments are independent of $\scriptF_{\tau_k}$, we have $\Delta\vct w_k \perp\!\!\!\perp \scriptF_{\tau_k}$ (and hence $\Delta\vct w_k$ is independent of $\vct x_k$).
%Moreover, since $\Delta\vct w_k$ is independent of $\scriptF_{\tau_k}$, it is independent of $\vct x_k$
%
Furthermore, define the averaged (projected) integrand
\begin{equation}~\label{eq:Hbar}
    \hspace{-0.3cm}\bar H_{k}(\vct x_k,\tilde{\vct u}_k)
    \coloneqq \frac{1}{\Delta\tau_k}\int_{\tau_k}^{\tau_{k+1}} H_{k}(\tau;\vct x_k,\tilde{\vct u}_k)\,\d\tau\in\R^{n\times d}.
\end{equation}
We approximate the It\^{o} integral in \eqref{eq:exact-discrete} by the constant-integrand It\^{o} integral as
\begin{align}
    \int_{\tau_k}^{\tau_{k+1}} H_{k}(\tau;\vct x_k,\tilde{\vct u}_k)\,\d \vct w_\tau
    &\approx \int_{\tau_k}^{\tau_{k+1}} \bar H_{k}(\vct x_k,\tilde{\vct u}_k)\,\d \vct w_\tau \nonumber \\
    &= \bar H_{k}(\vct x_k,\tilde{\vct u}_k)\,\Delta \vct w_{k}. \label{eq:proj-approx}
\end{align}
Note that $\bar H_k(\cdot)\in\R^{n\times d}$ is column-wise affine in $(\vct x_k,\tilde{\vct u}_k)$, i.e., for each channel $i$,
\begin{equation}\label{eq:Hbar-affine-cols}
    \bar H_{k,i}(\vct x_k,\tilde{\vct u}_k)
    =
    \tilde A_{k}^{(i)} \vct x_k + \tilde F_{k}^{(i)} \tilde{\vct u}_k + \tilde d_{k}^{(i)}.
\end{equation}
%The corresponding matrices and vectors are given in Appendix~\ref{app:exact-disc}.
Combining Assumption~\ref{assum:freeze-diff} with \eqref{eq:proj-approx}, we obtain the tractable discrete-time approximation of \eqref{eq:exact-discrete}
\begin{equation}\label{eq:approx-discrete-general}
    \boxed{
    \vct x_{k+1}
    \approx
    A_k \vct x_k + \vct{\nu}_k
    + \bar H_{k}(\vct x_k,\tilde{\vct u}_k)\,\Delta \vct w_{k}.
    }
\end{equation}

The approximation \eqref{eq:proj-approx} is optimal (in conditional mean-square) within the class of constant integrands; see Appendix~\ref{app:proj-opt}.
Naturally, one may wonder how much error the above approximation scheme incurs compared to exact discretization.
We now formalize one-step mean-square error bounds for (i) freezing $\vct x_\tau\mapsto \vct x_k$ inside the diffusion integrand, and (ii) projecting the frozen It\^{o} integral onto $\bar H_k\,\Delta\vct w_k$.

\begin{thm}[One-step diffusion discretization error]\label{thm:one-step-error}
Fix the time step $k\in\{0,\ldots,N-1\}$ and consider the exact mild update $\vct x_{k+1}^{\rm ex}$ in \eqref{eq:exact-discrete}.
Let
\begin{equation*}
    \vct x_{k+1}^{\rm fr} \coloneqq A_k \vct x_k + \vct{\nu}_k + \int_{\tau_k}^{\tau_{k+1}} H_k(\tau;\vct x_k,\tilde{\vct u}_k)\,\d \vct w_\tau,
\end{equation*}
denote the \textit{frozen} state update, and let $\vct x_{k+1}^{\rm ap}$ denote the \textit{projected} (fully approximated) state update in \eqref{eq:approx-discrete-general}.
Denote the associated errors $\vct e_{k+1}^{(x)}\coloneqq \vct x_{k+1}^{\rm ex}-\vct x_{k+1}^{\rm fr}$ (freezing) and $\vct e_{k+1}^{(p)}\coloneqq \vct x_{k+1}^{\rm fr}-\vct x_{k+1}^{\rm ap}$ (projection).
Assume Assumption~\ref{assum:local-C1} (in Appendix~\ref{app:assums-and-proofs}) holds on $\mathcal{T}_k=[\tau_k,\tau_{k+1}]$.
Then the conditional mean-square errors satisfy
\begin{subequations}\label{eq:thm-bounds}
\begin{align}
    \E_k\!\left[\|\vct e_{k+1}^{(x)}\|^2\right]
    &\lesssim \Delta\tau_k^{2}\left(1+\|\vct x_k\|^2+\|\tilde{\vct u}_k\|^2\right), \label{eq:freeze-bound}\\
    \E_k\!\left[\|\vct e_{k+1}^{(p)}\|^2\right]
    &\le
    \frac{1}{12}\,L_{H,k}(\vct x_k,\tilde{\vct u}_k)^2\,\Delta\tau_k^{3}, \label{eq:proj-bound}
\end{align}
\end{subequations}
where $L_{H,k}(\vct x_k,\tilde{\vct u}_k)$ is the one-step time-Lipschitz modulus of the map $\tau\mapsto H_k(\tau;\vct x_k,\tilde{\vct u}_k)$.
\end{thm}
\begin{proof}
    See Appendix~\ref{app:proof-thm1}.
\end{proof}

\begin{rem}
For small steps, the freezing error is $O(\Delta\tau_k^2)$ while the projection error is $O(\Delta\tau_k^3)$.
Hence, the freezing error dominates the diffusion discretization error in the present scheme.
Compared to existing methods~\cite{Ridderhof2019iCS,Kumagai2025Cislunar} that freeze the diffusion at the current reference \textit{before} linearization, which incurs an error $O(\Delta\tau)$, the proposed method which performs exact first-order linearization but approximate discretization has an error an order-of-magnitude smaller.
\end{rem}

\subsection{Moment Propagation}\label{subsec:moment-prop}
To encode the mean and covariance constraints \eqref{eq:term-cond} as well as reformulate the chance constraints \eqref{eq:ct-cc}, it is necessary to propagate the mean and covariance of the state.
We first expand the discretized dynamics \eqref{eq:approx-discrete-general} into a more familiar form to facilitate the analysis.
To this end, note the following equivalence
\begin{align}
    \vct x_{k+1}
    &= A_k \vct x_k + \vct{\nu}_k + \sum_{i=1}^{d} \big(\tilde{A}_k^{(i)} \vct x_k + \tilde{\vct{\nu}}_k^{(i)}\big) \Delta\vct w_k^{(i)} \nonumber\\
    &= \vct{\mathcal A}_k \vct x_k + \vct{\mathcal F}_k \tilde{\vct u}_k + \vct{\mathcal D}_k,\label{eq:lin-system}
\end{align}
where 
$
\vct{\mathcal A}_k \coloneqq A_k + \sum_{i=1}^{d} \tilde{A}_k^{(i)} \Delta\vct w_k^{(i)}, \,
\vct{\mathcal F}_k \coloneqq F_k + \sum_{i=1}^{d} \tilde{F}_k^{(i)} \Delta\vct w_k^{(i)}
$, and
$
\vct{\mathcal D}_k \coloneqq d_k + \sum_{i=1}^{d} \tilde{d}_k^{(i)} \Delta\vct  w_k^{(i)}.
$
Structurally, the resultant dynamical system \eqref{eq:lin-system} is that of a linear system perturbed by (correlated) multiplicative and additive noise.
Next, and similar to the derivation in \cite{Knaup2023ParametricCS}, we compute the first two moments of the uncertain affine system \eqref{eq:lin-system}.

\begin{thm}[One-step moment propagation]\label{thm:moment-prop}
Let $\bar{x}_k \coloneqq \E[\vct x_k]$ denote the mean state and let
\begin{align*}
\Sigma_{x_k} &\coloneqq \E\!\left[(\vct x_k - \bar{x}_k)(\vct x_k - \bar{x}_k)^\intercal\right],\\
\Sigma_{x_k \tilde{u}_k} &\coloneqq \E\!\left[(\vct x_k - \bar{x}_k)(\tilde{\vct u}_k - \E[\tilde{\vct u}_k])^\intercal\right],\\
\Sigma_{\tilde{u}_k} &\coloneqq \E\!\left[(\tilde{\vct u}_k - \E[\tilde{\vct u}_k])(\tilde{\vct u}_k - \E[\tilde{\vct u}_k])^\intercal\right]
\end{align*}
denote the state covariance, state--control cross covariance, and control covariance, respectively.
Then the one-step mean and covariance propagation associated with \eqref{eq:lin-system} is
\begin{subequations}\label{eq:moment-prop}
\begin{align}
    \bar{x}_{k+1} &= A_k \bar{x}_k + F_k \E[\tilde{\vct u}_k] + d_k, \label{eq:mean-prop}\\
    \Sigma_{x_{k+1}}
    \hspace{-3pt}&= A_k \Sigma_{x_k} A_k^\intercal
    \hspace{-2pt}+\hspace{-2pt} A_k \Sigma_{x_k \tilde{u}_k} F_k^\intercal
    \hspace{-2pt}+\hspace{-2pt} F_k \Sigma_{x_k \tilde{u}_k}^\intercal A_k^\intercal
    \hspace{-2pt}+\hspace{-2pt} F_k \Sigma_{\tilde{u}_k} F_k^\intercal \nonumber\\
    &\hspace{-0.5cm}
    + \Delta\tau_k \sum_{i=1}^{d}\Big(
    \tilde A_k^{(i)} \Sigma_{x_k} (\tilde A_k^{(i)})^\intercal
    + \tilde A_k^{(i)} \Sigma_{x_k \tilde{u}_k} (\tilde F_k^{(i)})^\intercal \nonumber\\
    &\hspace{-0.5cm}
    + \tilde F_k^{(i)} \Sigma_{x_k \tilde{u}_k}^\intercal (\tilde A_k^{(i)})^\intercal 
    + \tilde F_k^{(i)} \Sigma_{\tilde{u}_k} (\tilde F_k^{(i)})^\intercal
    + q_k^{(i)}(q_k^{(i)})^\intercal
    \Big), \label{eq:cov-prop}
\end{align}
\end{subequations}
where $q_k^{(i)} \coloneqq \tilde{A}_k^{(i)}\bar{x}_k + \tilde{F}_k^{(i)}\E[\tilde{\vct u}_k] + \tilde{d}_k^{(i)}$.
\end{thm}
\begin{proof}
See Appendix~\ref{app:proof-moment-prop}.
\end{proof}

\subsection{Policy parameterization}\label{subsec:policy}
We employ affine state-deviation feedback of the form
\begin{equation}\label{eq:affine-feedback}
    \tilde{\vct u}_k = \tilde v_k + \tilde K_k(\vct x_k-\bar{x}_k),
\end{equation}
where $\tilde v_k\coloneqq [v_k; \sigma_k]\in\R^{m+1}$ and $\tilde K_k\coloneqq [K_k; 0_{1\times n}]\in\R^{(m+1)\times n}$.
Under the parameterization \eqref{eq:affine-feedback}, the statistics of the control input satisfy $\E[\tilde{\vct u}_k]=\tilde v_k$, $\Sigma_{\tilde{u}_k} = \tilde K_k \Sigma_{x_k} \tilde K_k^\intercal$, and $\Sigma_{x_k \tilde{u}_k} = \Sigma_{x_k}\tilde K_k^\intercal$.
Substituting into \eqref{eq:moment-prop} yields
\begin{subequations}~\label{eq:moment-prop-affine}
\begin{align}
    \bar{x}_{k+1} 
    &= A_k \bar{x}_k + F_k \tilde{v}_k + d_k, \label{eq:mean-prop-affine-control}\\
    \Sigma_{x_{k+1}}
    &= (A_k + F_k \tilde K_k)\,\Sigma_{x_k}\,(A_k + F_k \tilde K_k)^\intercal + \Delta\tau_k\cdot \nonumber\\
    &\hspace{-0.9cm}\sum_{i=1}^{d}\Big(
        (\tilde A_k^{(i)} + \tilde F_k^{(i)}\tilde K_k)\,\Sigma_{x_k}\,(\tilde A_k^{(i)} + \tilde F_k^{(i)}\tilde K_k)^\intercal + q_k^{(i)}(q_k^{(i)})^\intercal
    \Big). \label{eq:cov-prop-affine-control}
\end{align}
\end{subequations}
From the resulting moment propagation in \eqref{eq:moment-prop-affine}, it is apparent that the mean dynamics \eqref{eq:mean-prop-affine-control} are linear in the decision variables $\{\bar{x}_k, v_k, \sigma_k\}$.
Equation \eqref{eq:cov-prop-affine-control}, on the other hand, makes explicit the nonlinear coupling between $\Sigma_{x_k}$ and $\tilde K_k$.
This issue will be resolved in Section~\ref{subsec:convex-reduction} through a suitable lossless relaxation.

\subsection{Chance Constraints}~\label{subsec:chance-constraints}
Because the exact diffusion linearization induces multiplicative noise, the discretized local model \eqref{eq:approx-discrete-general} is generally non-Gaussian even for Gaussian initial conditions. We therefore enforce chance constraints through a sufficient distributionally robust surrogate based only on the propagated first two moments \eqref{eq:moment-prop}, namely the \textit{distributionally robust} chance constraints (DR-CC)
\begin{subequations}~\label{eq:dr-ct-cc}
    \begin{align}
        \inf_{\P_{x_\tau} \in \scriptC(\mu_\tau, \Sigma_{x_\tau})} \P_{x_\tau}(\vct x_t \in \scriptX, \ \forall \tau \in [0,1]) \geq 1 - \Delta_x, \\ 
        \inf_{\P_{u_\tau} \in \scriptC(v_\tau, \Sigma_{u_\tau})} \P_{u_\tau}(\vct u_\tau \in \scriptU, \ \forall \tau \in [0,1]) \geq 1 - \Delta_u,
    \end{align}
\end{subequations}
where $\scriptC(\mu,\Sigma)$ denotes the Chebyshev ambiguity set of all distributions whose first two moments are equal to $(\mu, \Sigma)$.

We enforce \eqref{eq:dr-ct-cc} at the grid nodes $\tau_k$ (rather than continuously in time); continuous-time DR-CC enforcement via isoperimetric reformulations is possible~\cite{Elango2025ctSCvx} but omitted here. Using the polytopic sets in \eqref{eq:half-spaces}, we impose
\begin{subequations}~\label{eq:joint-dt-dr-cc}
\begin{align}
    &\inf_{\P_{x_k} \in \scriptC(\mu_k, \Sigma_{x_k})} \P_{x_k}(\alpha_j^\intercal \vct x_k + \beta_j \leq 0, \ \forall j,k) \geq 1 - \Delta_x, \\ 
    &\inf_{\P_{u_k} \in \scriptC(v_k, \Sigma_{u_k})} \P_{u_k}(a_j^\intercal \vct u_k + b_j \leq 0, \ \forall j,k) \geq 1 - \Delta_u.
\end{align}
\end{subequations}
The \textit{joint} DR-CC in \eqref{eq:joint-dt-dr-cc} are non-convex and intractable, in general~\cite{Lew2020ccSCP}.
We employ a standard relaxation using Boole's inequality~\cite{Prekopa1988Boole} to reformulate the joint DR-CC as the \textit{individual} DR-CC
\begin{subequations}~\label{eq:individual-dt-dr-cc}
    \begin{align}
        &\inf_{\P_{x_k} \in \scriptC(\mu_k, \Sigma_{x_k})} \P_{x_k}(\alpha_j^\intercal \vct x_k + \beta_j \leq 0) \geq 1 - \delta_{j,k}^{x}, \\ 
        &\inf_{\P_{u_k} \in \scriptC(v_k, \Sigma_{u_k})} \P_{u_k}(a_j^\intercal \vct u_k + b_j \leq 0) \geq 1 - \delta_{j,k}^{u}, \\ 
        &\sum_{j=1}^{N_x}\sum_{k=1}^{N} \delta_{j,k}^{x} \leq \Delta_{x},\quad \sum_{j=1}^{N_u}\sum_{k=0}^{N-1} \delta_{j,k}^{u} \leq \Delta_{u},
    \end{align}
\end{subequations}
where $\delta_{j,k}^{x}, \delta_{j,k}^{u} \in (0,0.5]$ represent the state and input risk allocations, respectively.
In general, these variables are unknown and must be optimized jointly with that of the state and control variables at run-time.
However, in this work, we employ the standard approximation of fixing the risk variables to the constant allocation $\delta_{j,k}^{x} = \Delta_x / (NN_x)$ and $\delta_{j,k}^{u} = \Delta_u / (NN_u)$.
We remark, however, that it is possible to jointly optimize the risk allocation with the controller using a two-step approach~\cite{Ono2008IRA,Pilipovsky2021IRA}.
Finally, we use results from the distributionally robust optimization literature~\cite{CalafioreElGhaoui2006DRCCLP} to evaluate the minimization problems in \eqref{eq:individual-dt-dr-cc} in closed form, yielding
\begin{subequations}~\label{eq:dr-cc-nonconvex}
    \begin{align}
        &\alpha_j^\intercal \mu_k + \beta_j + \scriptQ(\delta_{j,k}^{x}) \sqrt{\alpha_j^\intercal \Sigma_{x_k} \alpha_j} \leq 0, \\ 
        &a_j^\intercal v_k + b_j + \scriptQ(\delta_{j,k}^{u}) \sqrt{a_j^\intercal \Sigma_{u_k} a_j} \leq 0,
    \end{align}
\end{subequations}
where $\scriptQ(\delta) \triangleq \sqrt{\delta^{-1}(1-\delta)}$.

\subsection{Convex reduction}~\label{subsec:convex-reduction}
We now reformulate the non-convex constraints in the covariance propagation and distributionally robust (DR) chance constraints into a tractable convex form suitable for SCP subproblems. Specifically, the one-step covariance recursion \eqref{eq:cov-prop-affine-control} contains (i) nonlinear terms coupling the covariance and feedback variables $(\Sigma_{x_k},\tilde K_k)$, and (ii) quadratic outer-product terms coupling the mean and feedforward variables $(\bar x_k,\tilde v_k)$. In addition, the DR chance constraints \eqref{eq:dr-cc-nonconvex} are non-convex due to the square-root dependence on the state and input covariances. We address the covariance-recursion non-convexities via lifting and Schur-complement relaxations, and then convexify the chance constraints through auxiliary variables and first-order linearization. The resulting constraints are conic/semidefinite representable and therefore compatible with the SCP subproblems.

\subsubsection*{Nonlinearity in $(\Sigma_{x_k},\tilde K_k)$}
In a similar vein to \cite{Liu2025dtCS}, we introduce the change of variables
\begin{equation}\label{eq:U-def}
    U_k \coloneqq \tilde K_k \Sigma_{x_k} \in \R^{(m+1)\times n}.
\end{equation}
Whenever $\Sigma_{x_k}\succ 0$, we have $\tilde K_k = U_k \Sigma_{x_k}^{-1}$ and hence
\begin{equation}\label{eq:KSK}
    \tilde K_k \Sigma_{x_k}\tilde K_k^\intercal = U_k \Sigma_{x_k}^{-1} U_k^\intercal,
\end{equation}
which is jointly nonconvex in $(U_k,\Sigma_{x_k})$.
We convexify via the LMI relaxation
\begin{equation}\label{eq:Y-relax}
    Y_k \succeq U_k \Sigma_{x_k}^{-1} U_k^\intercal.
\end{equation}
Using the Schur complement (and the fact that $\Sigma_{x_k}\succ 0$), \eqref{eq:Y-relax} is equivalent to
\begin{equation}\label{eq:Y-schur}
    \begin{bmatrix}
        Y_k & U_k \\
        U_k^\intercal & \Sigma_{x_k}
    \end{bmatrix} \succeq 0.
\end{equation}

\subsubsection*{Bilinearity in $(\bar x_k,\tilde v_k)$}
For each $(i,k)$, define the auxiliary decision variable
\begin{equation}\label{eq:Sigma-tilde-relax}
    \tilde\Sigma_{ik} \succeq \tilde q_k^{(i)}(\tilde q_k^{(i)})^\intercal,
    \qquad
    \tilde q_k^{(i)} = \tilde{A}_k^{(i)} \bar{x}_k + \tilde{F}_k^{(i)} \tilde{v}_k + \tilde{d}_k^{(i)}.
\end{equation}
Again by the Schur complement, \eqref{eq:Sigma-tilde-relax} is equivalent to the LMI
\begin{equation}\label{eq:Sigma-tilde-schur}
    \begin{bmatrix}
        \tilde\Sigma_{ik} & \tilde q_k^{(i)} \\
        (\tilde q_k^{(i)})^\intercal & 1
    \end{bmatrix}\succeq 0.
\end{equation}

Substituting \eqref{eq:U-def}--\eqref{eq:Sigma-tilde-schur} into the covariance propagation \eqref{eq:cov-prop-affine-control} yields the following convex constraint system
\begin{subequations}\label{eq:cov-final-relaxed}
\begin{align}
    &\Sigma_{x_{k+1}}  = A_k \Sigma_{x_k} A_k^\intercal + A_k U_k^\intercal F_k^\intercal + F_k U_k A_k^\intercal + F_k Y_k F_k^\intercal \nonumber\\
    &\quad + \Delta\tau_k \sum_{i=1}^{d}\Big(
        \tilde{A}_k^{(i)} \Sigma_{x_k} (\tilde{A}_k^{(i)})^\intercal
        + \tilde{A}_k^{(i)} U_k^\intercal (\tilde{F}_k^{(i)})^\intercal \nonumber\\
        &\hspace{1cm}
        + \tilde{F}_k^{(i)} U_k (\tilde{A}_k^{(i)})^\intercal + \tilde{F}_k^{(i)} Y_k (\tilde{F}_k^{(i)})^\intercal + \tilde{\Sigma}_{ik}
    \Big), \label{eq:cov-final-relaxed-1}\\
    &\begin{bmatrix}
        Y_k & U_k \\
        U_k^\intercal & \Sigma_{x_k}
    \end{bmatrix} \succeq 0, \label{eq:cov-final-relaxed-2}\\
    &\begin{bmatrix}
        \tilde\Sigma_{ik} & \tilde{A}_k^{(i)} \bar{x}_k + \tilde{F}_k^{(i)} \tilde{v}_k + \tilde{d}_k^{(i)} \\
        \star & 1
    \end{bmatrix} \succeq 0, \quad \forall i\in\{1,\dots,d\},\label{eq:cov-final-relaxed-3}
\end{align}
\end{subequations}
for all time steps $k = 0,\ldots, N-1$.
The constraints in \eqref{eq:cov-final-relaxed} are affine in the decision variables $\{\bar x_k,\Sigma_{x_k},\tilde v_k,U_k,Y_k,\tilde\Sigma_{ik}\}$.
It is natural to ask whether the relaxations \eqref{eq:Y-relax} and \eqref{eq:Sigma-tilde-relax} introduce conservatism in the covariance propagation and, consequently, in the terminal covariance constraints \eqref{eq:term-cond}. In general, \eqref{eq:cov-final-relaxed-1} over-approximates the one-step covariance. Let $P_k$ denote the true state covariance under the linearized dynamics \eqref{eq:cov-prop-affine-control}. Then $\Sigma_{x_k}\succeq P_k$ for all $k=0,\ldots,N$ (with equality at $k=0$), so enforcing $\Sigma_{x_N}=\Sigma_f$ is sufficient to guarantee $P_N\preceq \Sigma_f$. Moreover, with a suitable regularization of the objective \eqref{eq:obj-riemann}, the relaxations become lossless at optimality; i.e., the relaxed and original covariance-steering problems share the same optimal decision variables and objective value (see Theorem~\ref{thm:cov-losslessness}, Appendix~\ref{app:lossless}).

Next, we convexify the DR chance constraints \eqref{eq:dr-cc-nonconvex}. We first relax the input-covariance terms by replacing $\Sigma_{u_k}$ with $Y_k$. Since $Y_k \succeq \Sigma_{u_k}$, this yields sufficient (generally conservative) input constraints. We then introduce auxiliary variables $\kappa_{x,j,k}, \kappa_{u,j,k}\ge 0$ satisfying
\[
\kappa_{x,j,k}^2 \ge \alpha_j^\intercal \Sigma_{x_k}\alpha_j,
\qquad
\kappa_{u,j,k}^2 \ge a_j^\intercal Y_k a_j.
\]
Under this lifting, the DR chance constraints become linear in $(\mu_k,v_k,\kappa_{x,j,k},\kappa_{u,j,k})$, while the remaining non-convexity is isolated in the quadratic inequalities defining $\kappa_{x,j,k}$ and $\kappa_{u,j,k}$. Given reference values $(\hat\kappa_{x,j,k},\hat\kappa_{u,j,k})$ (initialized as discussed in Section~\ref{sec:scp}), we linearize these inequalities to obtain
\begin{subequations}~\label{eq:dr-cc-lin}
    \begin{align}
        \alpha_j^\intercal \bar{x}_k + \beta_j + \scriptQ(\delta_{j,k}^{x}) \kappa_{x,j,k} &\leq 0, \label{eq:dr-cc-lin1}\\
        \alpha_j^\intercal \Sigma_{x_k} \alpha_j - \hat\kappa_{x,j,k}^2 - 2\hat\kappa_{x,j,k}(\kappa_{x,j,k} - \hat\kappa_{x,j,k}) &\leq 0, \label{eq:dr-cc-lin2}\\
        a_j^\intercal v_k + b_j + \scriptQ(\delta_{j,k}^{u}) \kappa_{u,j,k} &\leq 0, \label{eq:dr-cc-lin3}\\
        a_j^\intercal Y_k a_j - \hat\kappa_{u,j,k}^2 - 2\hat\kappa_{u,j,k}(\kappa_{u,j,k} - \hat\kappa_{u,j,k}) &\leq 0. \label{eq:dr-cc-lin4}
    \end{align}
\end{subequations}

Finally, for completeness we note that the initial and terminal distributional constraints on the state are convex by design, since we assume the moments are decision variables in the resulting program.
As a result, the boundary moment constraints are simply written as
\begin{subequations}~\label{eq:boundary-cond-final}
    \begin{align}
        \bar{x}_0 &= \mu_i, \hspace{-1cm}&&\Sigma_{x_0} = \Sigma_i, \\ 
        \bar{x}_N &= \mu_f, \hspace{-1cm}&&\Sigma_{x_N} = \Sigma_f.
    \end{align}
\end{subequations}

%%%%%%%%%%%%%%%%%%%%%%%%%%%%%%%%%%%%%%%%%%%%%%%%%%%%%%%%%%%%%%%%%%%%%%%%%%%%%%%%
\section{SEQUENTIAL CONVEX PROGRAMMING}~\label{sec:scp}
In this section, we present the full FFT-iCS SCP procedure used to compute a local solution of Problem~\ref{prob:FFT-CS}. At iteration $\ell$, the nonlinear time-scaled SDE is linearized and discretized about a current reference, the resulting convex penalty covariance-steering subproblem is solved, and the candidate step is accepted or rejected according to the standard SCvx$^\star$ ratio test. In the following, we state the initialization, diagnostics, and update formulas explicitly.

The first step is to construct an initial reference trajectory and time-dilation profile $(\hat x^{(0)},\hat u^{(0)},\hat\sigma^{(0)})$ for the system linearization. This initialization is problem dependent; in practice, any boundary-consistent deterministic or mean-only steering solution, or any near-feasible interpolating trajectory with admissible control and mesh variables, may be used. Given $(\hat x^{(0)},\hat u^{(0)},\hat\sigma^{(0)})$, the remaining chance-constraint linearization variables are initialized from a preliminary \emph{unconstrained} FFT-CS solve\footnote{This can be interpreted as a zeroth (warm-start) iteration. In practice, the corresponding optimal decision variables may also be passed as the initial guess for the first chance-constrained solve, which often improves solver convergence.}, exactly as in \cite{Kumagai2025Cislunar}. From the resulting covariances $\{\Sigma_{x_k}^{\star},Y_k^{\star}\}$, we set
\begin{equation*}
    \hat\kappa_{x,j,k}^{(0)} = \sqrt{\alpha_j^\intercal \Sigma_{x_k}^{\star} \alpha_j},
    \qquad
    \hat\kappa_{u,j,k}^{(0)} = \sqrt{a_j^\intercal Y_k^{\star} a_j}.
\end{equation*}
Together with the remaining primal variables produced by the warm start, this yields the full reference tuple $\hat z^{(0)}$ for Problem~\ref{prob:penalty-CS-problem}. We denote by
\begin{equation*}
    \hat y^{(\ell)} \coloneqq \big(\hat x^{(\ell)},\hat u^{(\ell)},\hat\sigma^{(\ell)},\hat\kappa^{(\ell)}\big)
\end{equation*}
the subset of $\hat z^{(\ell)}$ required for linearization and discretization. The associated local model data are summarized by
\begin{equation*}
    \mathcal Z^{(\ell)} = \mathrm{dis}_{\scriptP} \circ \mathrm{lin}_{\scriptD}(\hat y^{(\ell)}),
\end{equation*}
where $\mathcal D$ denotes the time-scaled nonlinear system \eqref{eq:nl-sde-scaled} and $\scriptP$ the partition in \eqref{eq:approx-discrete-general}.

There are two well-known pathologies of SCP subproblems: (i) artificial \emph{infeasibility} and (ii) artificial \emph{unboundedness}. To address the former, we introduce virtual controls / buffer variables in the linearized constraints and penalize them in the objective. To address the latter, we impose trust-region constraints around the current linearization point.
Specifically, we soften the linearized mean dynamics and the linearized chance-constraint quadratic surrogates via
\begin{subequations}
\begin{align}
    \bar{x}_{k+1} - A_k \bar{x}_k - F_k \tilde{v}_k - d_k &= \xi_k, \label{eq:linearized-mean-eq-buffer}\\
    \alpha_j^\intercal\Sigma_{x_k} \alpha_j - 2\hat\kappa_{x,j,k}\kappa_{x,j,k} + \hat\kappa_{x,j,k}^2 &\leq \zeta_{x,j,k}, \label{eq:linearized-chance-ineq-buffer1} \\
    a_j^\intercal Y_k a_j - 2\hat\kappa_{u,j,k}\kappa_{u,j,k} + \hat\kappa_{u,j,k}^2 &\leq \zeta_{u,j,k}, \label{eq:linearized-chance-ineq-buffer2}
\end{align}
\end{subequations}
where $\xi_k\in\R^n$ and $\zeta_{x,j,k},\zeta_{u,j,k}\ge 0$ are penalized using the augmented-Lagrangian penalty \cite{Oguri2023ALSCvx}
\begin{equation}
    \scriptJ_{\mathrm{pen}}(\xi,\zeta\mid w,\lambda,\mu)
    = \mu^\intercal \xi + \frac{w}{2}\|\xi\|_2^2 + \lambda^\intercal \zeta + \frac{w}{2}\|[\zeta]_+\|_2^2,
\end{equation}
with multipliers $\mu\in\R^{Nn}$ and $\lambda\in\R^{(N_x+N_u)N}$ and penalty weight $w>0$.
Artificial unboundedness is mitigated by the trust region
\begin{equation}~\label{eq:trust-region}
    \left\|
    \begin{bmatrix}
        W_{x}           && \\
        & W_{\tilde{u}} &  \\
        && W_{\kappa}
    \end{bmatrix}
    \begin{bmatrix}
        \bar{x} - \hat{x} \\
        \tilde{v} - [\hat{u};\hat{\sigma}] \\
        \kappa - \hat\kappa
    \end{bmatrix}
    \right\|_{\infty} \leq \Delta_{\mathrm{tr}},
\end{equation}
where $W_x, W_{\tilde u}, W_\kappa \succ 0$ are scaling matrices and $\Delta_{\mathrm{tr}}>0$ is the trust-region radius.
Putting these ingredients together, the convex penalty covariance-steering subproblem solved at each SCP iteration is the following.
\begin{problem}[P-CS]~\label{prob:penalty-CS-problem}
    Let the primal decision variables be
    \begin{align*}
        z \coloneqq \Big(&\{\bar x_k\}_{k=0}^{N},\,\{\Sigma_{x_k}\}_{k=0}^{N},\,\{\tilde v_k\}_{k=0}^{N-1},\,\{U_k,Y_k\}_{k=0}^{N-1}, \\
        &\{\tilde\Sigma_{ik}\}_{i=1,k=0}^{d,\;N-1},\,\{\kappa_{x,j,k}\}_{j=1,k=1}^{N_x,\;N},\,\{\kappa_{u,j,k}\}_{j=1,k=0}^{N_u,\;N-1}\Big),
    \end{align*}
    and define the penalty variables
    \begin{equation*}
        \xi := \{\xi_k\}_{k=0}^{N-1}, \quad
        \zeta := \big(\{\zeta_{x,j,k}\}_{j=1,k=1}^{N_x,\,N},\,\{\zeta_{u,j,k}\}_{j=1,k=0}^{N_u,\,N-1}\big).
    \end{equation*}
    Define $\mathbb{L} := \R^{n} \times \S_{++}^{n} \times \R^{m} \times \R_{+} \times \R^{\tilde{m}\times n} \times \S_{+}^{\tilde{m}} \times \S_{+}^{n} \times \R_{+}^{N_x} \times \R_{+}^{N_u}$ as the per-stage feasible cone, where $\tilde{m} = m+1$.
    The penalty covariance-steering problem is

    \noindent\hspace*{-2.1em}%
    \begin{minipage}{\linewidth}
    \begin{align*}
        \min_{z,\xi,\zeta}\quad
        & \eta \sum_{k=0}^{N-1}\Delta\tau_k\sigma_k + \scriptJ_{\mathrm{reg}}(z) + \scriptJ_{\mathrm{pen}}(\xi,\zeta) \\
        \text{s.t.}\quad\,
        & \eqref{eq:linearized-mean-eq-buffer} &&\hspace{-5cm}\gets\textrm{soft mean constraints} \\
        & \eqref{eq:cov-final-relaxed}         &&\hspace{-5cm}\gets\textrm{covariance constraints} \\
        & \eqref{eq:dr-cc-lin1},\, \eqref{eq:dr-cc-lin3},\, \eqref{eq:linearized-chance-ineq-buffer1},\, \eqref{eq:linearized-chance-ineq-buffer2}
        \ \rlap{$\gets\ \textrm{soft chance constraints}$} \\
        & \eqref{eq:trust-region}              &&\hspace{-5cm}\gets\textrm{trust region constraints} \\
        & \eqref{eq:boundary-cond-final}       &&\hspace{-5cm}\gets\textrm{boundary constraints} \\
        & z \in \mathbb{L},\; \xi \in \R^{Nn},\; \zeta \in \R_{+}^{(N_x + N_u)N}.
    \end{align*}
    \end{minipage}
\end{problem}

To define the SCvx$^\star$ diagnostics, let
\begin{subequations}\label{eq:scp-merit-fns}
\begin{align}
    &\scriptJ_{\mathrm{nl}}^{(\ell)}(z)
    \coloneqq \scriptJ(z) +  \scriptJ_{\mathrm{pen}}\big(g_{\mathrm{nl}}(z),h_{\mathrm{nl}}(z)\mid w^{(\ell)},\lambda^{(\ell)},\mu^{(\ell)}\big), \label{eq:scp-merit-fns-a}\\
    &\scriptL_{\mathrm{cvx}}^{(\ell)}(z,\xi,\zeta)
    \coloneqq \scriptJ(z) + \scriptJ_{\mathrm{pen}}\big(\xi,\zeta\mid w^{(\ell)},\lambda^{(\ell)},\mu^{(\ell)}\big), \label{eq:scp-merit-fns-b}
\end{align}
\end{subequations}
where, for all $k = 0,\ldots, N-1$, the nonlinear residual maps are given by
\begin{align*}
    g_{\mathrm{nl},k}(z)
    &\coloneqq \bar{x}_{k+1} - \left(\bar{x}_{k} + \int_{\tau_{k}}^{\tau_{k+1}} \sigma_{k} f(\bar{x}_{\tau},v_{k})\, \d\tau\right),\\
    h_{\mathrm{nl},x,j,k}(z)
    &\coloneqq \alpha_j^\intercal \Sigma_{x_{k+1}}\alpha_j - \kappa_{x,j,{k+1}}^2,
    &&\hspace{-2.5cm} j=1,\ldots,N_x,\\
    h_{\mathrm{nl},u,j,k}(z)
    &\coloneqq a_j^\intercal Y_k a_j - \kappa_{u,j,k}^{2},
    && \hspace{-2.5cm}j=1,\ldots,N_u,
\end{align*}
and $g_{\mathrm{nl}}(z)$ and $h_{\mathrm{nl}}(z)$ collect these equality and inequality residuals, respectively. Thus, after solving Problem~\ref{prob:penalty-CS-problem} at iteration $\ell$ and obtaining $(z^{\star},\xi^{\star},\zeta^{\star})$, we define the nonlinear cost improvement, linearized cost improvement, and nonlinear infeasibility by
\begin{subequations}\label{eq:scvx-diagnostics}
\begin{align}
    \Delta J^{(\ell)}
    &\coloneqq \scriptJ_{\mathrm{nl}}^{(\ell)}\big(\hat z^{(\ell)}\big) - \scriptJ_{\mathrm{nl}}^{(\ell)}\big(z^{\star}\big), \label{eq:scvx-diagnostics-a}\\
    \Delta L^{(\ell)}
    &\coloneqq \scriptJ_{\mathrm{nl}}^{(\ell)}\big(\hat z^{(\ell)}\big) - \scriptL_{\mathrm{cvx}}^{(\ell)}\big(z^{\star},\xi^{\star},\zeta^{\star}\big), \label{eq:scvx-diagnostics-b}\\
    \chi^{(\ell)}
    &\coloneqq \left\|\begin{bmatrix} g_{\mathrm{nl}}(z^{\star}) \\ [h_{\mathrm{nl}}(z^{\star})]_+ \end{bmatrix}\right\|_2. \label{eq:scvx-diagnostics-c}
\end{align}
\end{subequations}
A subtle but important point is that the definition \eqref{eq:scvx-diagnostics-b} uses the \emph{same} iteration-dependent augmented-Lagrangian quantities $\big(w^{(\ell)},\lambda^{(\ell)},\mu^{(\ell)}\big)$ in both terms. This mirrors the SCvx$^\star$ construction in \cite{Oguri2023ALSCvx} and preserves the key property $\Delta L^{(\ell)}\ge 0$; if one instead mixes different iterations' penalty parameters, then this monotonicity can be lost.

The step-acceptance ratio is then defined by
\begin{equation}\label{eq:scvx-ratio}
    \rho^{(\ell)} \coloneqq \Delta J^{(\ell)} / \Delta L^{(\ell)}
\end{equation}
and the candidate step is accepted whenever $\rho^{(\ell)} \ge \rho_0$. Convergence is declared when both optimality and feasibility have converged, namely,
\begin{equation}\label{eq:scvx-conv-criterion}
    |\Delta J^{(\ell)}| \le \epsilon_{\mathrm{opt}}
    \quad\cap\quad
    \chi^{(\ell)} \le \epsilon_{\mathrm{feas}}.
\end{equation}
When an accepted step also satisfies the asymptotic exactness condition
\begin{equation}\label{eq:scvx-delta-condition}
    |\Delta J^{(\ell)}| < \delta^{(\ell)},
\end{equation}
the augmented-Lagrangian quantities are updated using the \emph{nonlinear} residuals as
\begin{equation}\label{eq:scvx-al-update}
\begin{aligned}
    \mu^{(\ell+1)}
    &= \mu^{(\ell)} + w^{(\ell)} g_{\mathrm{nl}}(z^{\star}), \\
    \lambda^{(\ell+1)}
    &= \big[\lambda^{(\ell)} + w^{(\ell)} h_{\mathrm{nl}}(z^{\star})\big]_+, \\
    w^{(\ell+1)}
    &= \min\big\{\beta w^{(\ell)},\, w_{\max}\big\}, \\
    \delta^{(\ell+1)}
    &=
    \begin{cases}
        |\Delta J^{(\ell)}|, & \delta^{(\ell)} = \infty,\\
        \gamma \, \delta^{(\ell)}, & \text{otherwise},
    \end{cases} 
\end{aligned}
\end{equation}
where $\beta>1$ and $\gamma\in(0,1)$. If \eqref{eq:scvx-delta-condition} is not satisfied, then $\mu$, $\lambda$, $w$, and $\delta$ are left unchanged.

Finally, the trust-region radius is updated according to the standard SCvx rule
\begin{equation}\label{eq:scvx-tr-update}
    \Delta_{\mathrm{tr}}^{(\ell+1)} =
    \begin{cases}
        \max\big\{\Delta_{\mathrm{tr}}^{(\ell)}/\alpha_1,\, \Delta_{\min}\big\}, & \rho^{(\ell)} < \rho_1,\\
        \Delta_{\mathrm{tr}}^{(\ell)}, & \rho_1 \le \rho^{(\ell)} < \rho_2,\\
        \min\big\{\alpha_2\Delta_{\mathrm{tr}}^{(\ell)},\, \Delta_{\max}\big\}, & \rho^{(\ell)} \ge \rho_2,
    \end{cases}
\end{equation}
where $\rho_0 < \rho_1 < \rho_2$ and $\alpha_1,\alpha_2 > 1$.

The complete FFT-iCS SCvx$^\star$ procedure is summarized in Algorithm~\ref{alg:scvxstar}.
\algrenewcommand\algorithmicindent{1.0em}
\begin{algorithm}[!htb]
\caption{FFT-iCS \texttt{SCvx*}}
\label{alg:scvxstar}
\begin{algorithmic}[1]
\Statex \textbf{Input:} $\epsilon_{\mathrm{opt}},\epsilon_{\mathrm{feas}},\hat z^{(0)},\hat y^{(0)},\Delta_{\mathrm{tr}}^{(0)},\Delta_{\min},\Delta_{\max},w^{(0)},$
\Statex \hspace{\algorithmicindent} $w_{\max},\rho_0,\rho_1,\rho_2,\alpha_1,\alpha_2,\beta,\gamma$
\Statex \textbf{Output:} local solution $(v^\star, K^\star)$
\State $\ell \gets 0$, $\Delta J^{(0)} \gets \infty$, $\chi^{(0)} \gets \infty$, $\delta^{(0)} \gets \infty$, $\lambda^{(0)} \gets 0$, $\mu^{(0)} \gets 0$
\While{not converged \textbf{and} $\ell < \ell_{\max}$} \hspace{1.2cm}\Comment{Eq. \eqref{eq:scvx-conv-criterion}}
    \State $\mathcal Z^{(\ell)} \gets \mathrm{dis}_{\scriptP}\circ\mathrm{lin}_{\scriptD}(\hat y^{(\ell)})$ 
    \State $\{z^{\star},\xi^{\star},\zeta^{\star}\} \gets$ solve Problem~\ref{prob:penalty-CS-problem}
    \State Compute $\{\Delta J^{(\ell)},\Delta L^{(\ell)},\chi^{(\ell)}\}$ \hspace{2cm} \Comment{Eq. \eqref{eq:scvx-diagnostics}}
    \State Set $\hat z^{(\ell+1)} \gets \hat z^{(\ell)}$, $\lambda^{(\ell+1)} \gets \lambda^{(\ell)}$, $\mu^{(\ell+1)} \gets \mu^{(\ell)}$, $w^{(\ell+1)} \gets w^{(\ell)}$, $\delta^{(\ell+1)} \gets \delta^{(\ell)}$
    \If{$\Delta L^{(\ell)} = 0$}
        \State $\rho^{(\ell)} \gets 1$
    \Else
        \State $\rho^{(\ell)} \gets \Delta J^{(\ell)} / \Delta L^{(\ell)}$ 
    \EndIf
    \If{$\rho^{(\ell)} \ge \rho_0$}
        \State $\hat z^{(\ell+1)} \gets z^{\star}$ \hspace{3.5cm} \Comment{accept the step}
        \State extract $\hat y^{(\ell+1)}$ from $\hat z^{(\ell+1)}$
        \If{$|\Delta J^{(\ell)}| < \delta^{(\ell)}$}
            \State Update $\{\mu,\lambda,w,\delta\}^{(\ell+1)}$ \hspace{2.2cm} \Comment{Eq. \eqref{eq:scvx-al-update}}
        \EndIf
    \EndIf
    \State Update $\Delta_{\mathrm{tr}}^{(\ell+1)}$ \hspace{4.2cm} \Comment{Eq. \eqref{eq:scvx-tr-update}}
    \State $\ell \gets \ell + 1$
\EndWhile
\State \Return $\hat z^{(\ell)}$
\end{algorithmic}
\end{algorithm}

Algorithm~\ref{alg:scvxstar} makes explicit the three quantities that are central to the SCP logic: the actual nonlinear improvement $\Delta J^{(\ell)}$, the predicted convex-model improvement $\Delta L^{(\ell)}$, and the nonlinear infeasibility $\chi^{(\ell)}$. The ratio $\rho^{(\ell)}$ governs step acceptance and trust-region adaptation, while the condition \eqref{eq:scvx-delta-condition} determines when the augmented-Lagrangian multipliers are advanced. In this way, the algorithm balances progress in optimality and feasibility while preserving the robustness properties of SCvx$^\star$ established in \cite{Oguri2023ALSCvx}.

%%%%%%%%%%%%%%%%%%%%%%%%%%%%%%%%%%%%%%%%%%%%%%%%%%%%%%%%%%%%%%%%%%%%%%%%%%%%%%%%
\section{NUMERICAL EXAMPLE}~\label{sec:examples}
We consider a double integrator system affected by both drag forces and a velocity-dependent disturbance.
Let the state be $\vct x := [\vct r; \vct v] \in \R^{2}$ in terms of position and velocity, and control input be $\vct u := a \in \R^{}$ in terms of the acceleration.
The dynamics are assumed to be
\begin{equation}~\label{eq:sde-double-int}
    \d \!
    \begin{bmatrix}
        \vct r_t \\ 
        \vct v_t
    \end{bmatrix} = 
    \begin{bmatrix}
        \vct v_t \\ 
        \vct a_t - C_D \vct v_t \|\vct v_t\|_{2}
    \end{bmatrix}\, \d t + 
    \begin{bmatrix}
        0 \\ 
        g_0 + g_1 \|\vct v_t\|_{2}
    \end{bmatrix}\, \d \vct w_{t},
\end{equation}
over the interval $t \in [0, t_f]$, with the final time $t_f$ unknown.
The coefficient $C_D = 0.15$ represents the nonlinear drag coefficient, and $g_0, g_1 > 0$ denote the base diffusion and state-dependent diffusion gains, which we will be given depending on the context.
For the discretization, we choose $N = 30$ steps with a uniform time-scaled partition $\scriptP = (\tau_0,\ldots,\tau_N)$ such that $\Delta\tau_k\equiv\Delta\tau = 0.02$.
The SCvx* parameters are given in Table~\ref{tab:scvx-params}, where we use the same tolerance $\epsilon = \epsilon_{\text{feas}} = \epsilon_{\text{opt}}$ for both the objective decrease and constraint feasibility, and $\ell_{\max} = 100$ maximum iterations.
\begin{table}[t]
    \caption{SCvx* parameters}
    \label{tab:scvx-params}
    \centering
    \small
    \setlength{\tabcolsep}{3pt}
    \begin{tabular}{c c c c}
        \toprule
        $\epsilon$ & $\{\rho_0,\rho_1,\rho_2\}$ & $\{\alpha_1,\alpha_2,\beta,\gamma,w_{\max}\}$ & $\{\Delta^{(0)},\Delta_{\min},\Delta_{\max}\}$ \\
        \midrule
        $10^{-5}$ & $\{0,0.25,0.7\}$ & $\{2,3,2,0.9,10^6\}$ & $\{0.1,10^{-10},10\}$ \\
        \bottomrule
    \end{tabular}
\end{table}
For the regularized objective in Problem~\ref{prob:penalty-CS-problem}, we choose
\begin{equation}~\label{eq:baseline-obj}
\scriptJ_{\mathrm{reg}} = \sum_{k=0}^{N-1}\Delta\tau\big(\mathrm{tr}(Q \Sigma_{x_k}) + \mathrm{tr}(R Y_k)\big) + \sum_{i=1}^{d}\epsilon_{\tilde\Sigma}\,\mathrm{tr}(W_{\tilde\Sigma}\tilde\Sigma_{ik}),
\end{equation}
where $\epsilon_{\tilde\Sigma}=10^{-4}$ and $W_{\tilde\Sigma}=I$ to satisfy the lossless-relaxation conditions (see Appendix~\ref{app:lossless}). We choose $Q=\textrm{diag}(10,1)$ and $R=0.1$ as the state and control covariance weights, respectively.
The initial state is chosen with mean $\mu_i = [0,0]^\intercal$ and covariance $\Sigma_i = 0.15 I_2$, and the desired final state is chosen with mean $\mu_f = [1, 0]^\intercal$ and covariance $\Sigma_f = \Sigma_i$.
We enforce the control chance constraints $|u| \leq u_{\max} = 5$, with $\Delta_u = 0.1$, distributed uniformly across all time steps and sides.
The convex SDP in Problem~\ref{prob:penalty-CS-problem} is solved using the \texttt{cvxpy} optimization suite~\cite{diamond2016cvxpy} with the \texttt{mosek} solver~\cite{mosek}.
For numerical stability, we enforce bounds on the time dilation factor as $0.4 = \sigma_{\min} \leq \sigma_k \leq \sigma_{\max} = 1.6$ for all $k = 0,\ldots, N-1$.
For Monte Carlo validation, we use $N_{\mathrm{mc}}=1,000$ rollouts of the nonlinear SDE \eqref{eq:sde-double-int} simulated by Milstein's method with $n_{\mathrm{sub}}=10$ sub-steps per normalized interval; see Appendix~\ref{app:sde-numerical-int} for the numerical integration details.

\subsection{Additive Noise}~\label{subsec:add-noise-analysis}
In this first example, we set $g_0=0.2$ and $g_1=0$ (purely additive noise) to isolate the effect of free-final-time optimization.
We sweep $\eta\in[0,10]$ to assess both time minimization and SCP convergence. Table~\ref{tab:oft-eta-sweep} shows convergence within the iteration budget for all tested values (reported as $\ell_{\mathrm{conv}}$) and a monotonic decrease in the optimal final time as $\eta$ increases.
\begin{table}[!htb]
\caption{Selected Optimal Free Final Times}
\label{tab:oft-eta-sweep}
\centering
\begin{tabular}{c c c c c c c c}
    \toprule
    $\eta$ & 0 & 0.2 & 0.5 & 0.8 & 1 & 2 & 10 \\
    \midrule
    $t_f^\star$ & 1.60 & 1.43 & 1.35 & 1.28 & 1.22 & 1.08 & 0.99 \\
    \midrule
    $\ell_{\mathrm{conv}}$ & 12 & 11 & 28 & 44 & 62 & 64 & 52 \\ 
    \bottomrule
\end{tabular}
\end{table}
The value $\eta=0$ corresponds to the standard CS objective~\cite{Liu2025dtCS} with no final-time penalty, so the optimizer drives the time-dilation variables to the \emph{maximum} attainable final time\footnote{With $\eta=0$, the final time is unpenalized; increasing $\sigma_k$ increases the available physical time per normalized interval, which weakly enlarges the feasible set and typically reduces the required control effort (and associated penalties). Hence, under the box constraints $\sigma_{\min} \le \sigma_k\le\sigma_{\max}$, the optimizer saturates $\sigma_k^\star=\sigma_{\max}$ for all $k$.
}
$\sigma_{\max}$ as the baseline solution.
Further increasing $\eta$ causes the optimal time to decrease beyond this point, as expected, until it tapers out at approximately one, which represents the minimum value the final time can decrease such that the problem still remains feasible (i.e., as $\eta\rightarrow\infty$).

To illustrate the convergence of the FFT-iCS algorithm, we choose a representative value $\eta = 1$ and show the decay of the nonlinear cost improvement $|\Delta J|$ as well as the nonlinear constraint infeasibility $\chi$, in Figure~\ref{fig:scvx-conv}.
\begin{figure}[!htb]
    \centering
    \includegraphics[width=\linewidth]{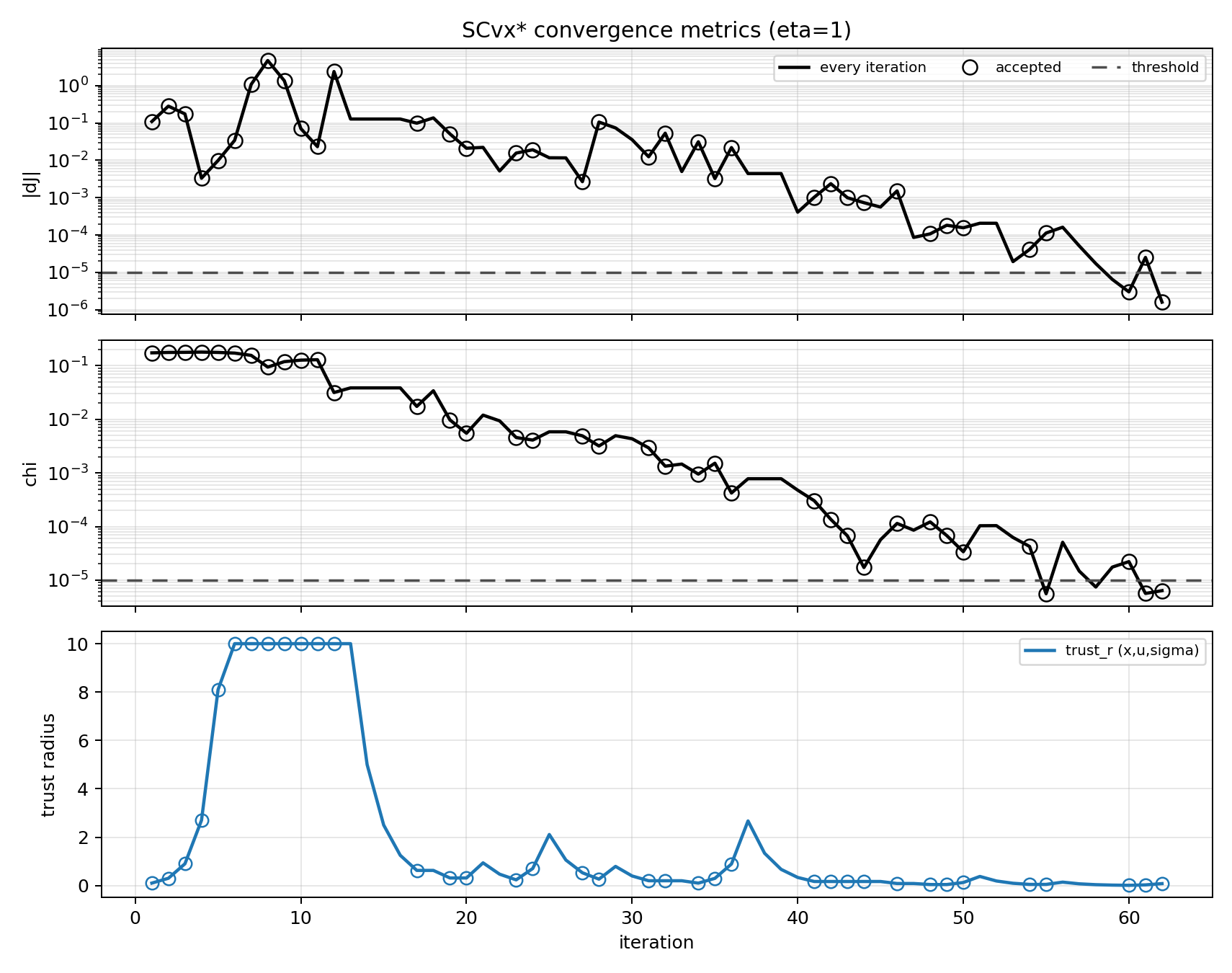}
    \caption{FFT-iCS convergence metrics.}
    \label{fig:scvx-conv}
\end{figure}
The SCvx* algorithm successfully converges in 62 iterations, and the converged solution satisfies both the state terminal covariance constraints as well as the control input chance constraints; Figure~\ref{fig:optimal-trajectories} shows the Monte Carlo state and control samples under the converged optimal policy $(K_k^\star, v_k^\star)$.
Even when simulated against the full non-linear SDE \eqref{eq:sde-double-int}, we find that the empirical terminal position variance is $\sigma_{p,N}^{\textrm{mc}} = 0.135 < \sigma_{p,f} = \sqrt{0.15}$ and the true worst-case control input risk over the entire horizon is $\max_{j,k} \delta_{j,k}^\star = 0.078 < \delta_{j,k} = 0.1$.
\begin{figure}[!htb]
    \centering
    \begin{subfigure}[t]{\linewidth}
        \centering
        \includegraphics[width=\linewidth]{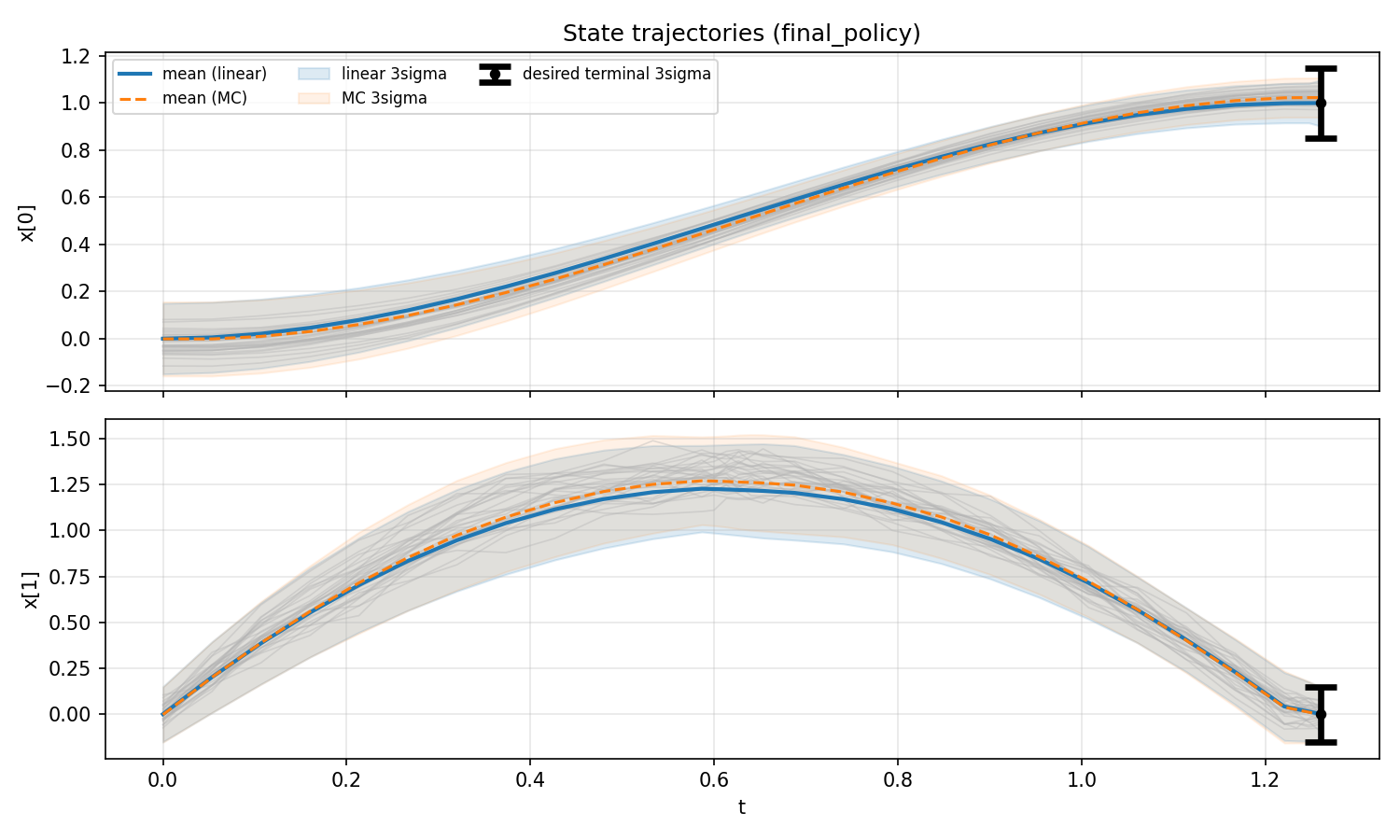}
        \caption{Mean state and $3\sigma$ state covariance trajectories with respect to linear (blue) and nonlinear (orange) models.}
        \label{fig:optimal-states}
    \end{subfigure}

    \vspace{0.5em}

    \begin{subfigure}[t]{\linewidth}
        \centering
        \includegraphics[width=\linewidth]{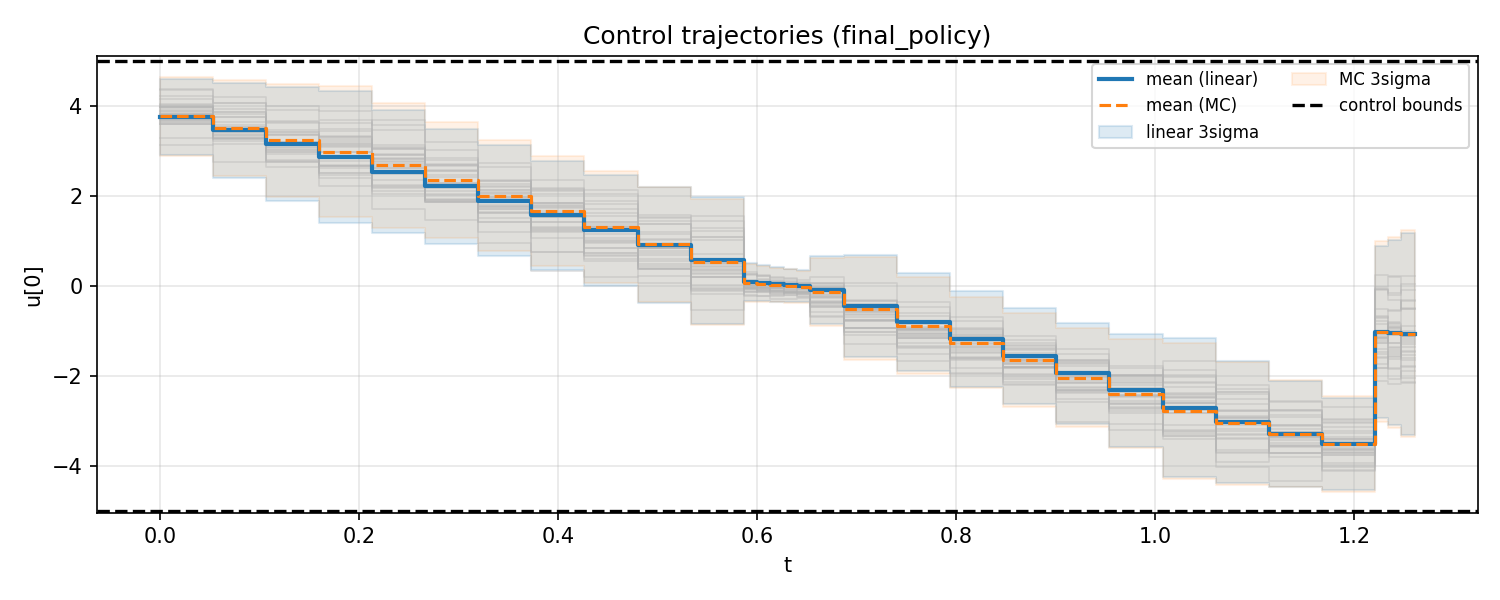}
        \caption{Mean control and $3\sigma$ control covariance trajectories with respect to linear (blue) and nonlinear (orange) models.}
        \label{fig:optimal-controls}
    \end{subfigure}
    \caption{Converged FFT-iCS optimal state and control trajectories for $\eta=1$.}
    \label{fig:optimal-trajectories}
\end{figure}

\subsection{Multiplicative Noise}~\label{subsec:mult-noise-analysis}
To illustrate the proposed \textit{full} diffusion linearization of the proposed approach, we now set a large multiplicative noise via $g_1 = 1$ and compare the resulting converged optimal trajectories with that of a \textit{frozen} diffusion linearization approach (e.g., \cite{Ridderhof2019iCS,Oguri2022iCS,Benedikter2022iCS}).
For these simulations, we change the SCvx* parameters to $\beta = 1.5,\, \Delta^{(0)} = 0.03,\, \Delta_{\max} = 0.5$, and both the full and frozen diffusion linearization methods converge in 16 and 33 iterations, respectively.
Figure~\ref{fig:frozen-full-3sigma} shows the evolution of the state covariance for both the (proposed) full diffusion linearization method as well the frozen diffusion method.
It is clear that even though both methods successfully steer the terminal state covariance with respect to the linearized system to its intended final value (all solid lines align at $t_f$), the discrepancy in the true nonlinear covariance (estimated through MC sample covariance) is much different.
Indeed, for the position, we have that $\sigma_{p,N}^{\textrm{full},\textrm{mc}} - \sigma_{p,f} = -0.002$ while $\sigma_{p,N}^{\textrm{frozen},\textrm{mc}} - \sigma_{p,f} = 0.04$, therefore the proposed method is not only nonlinear feasible at convergence, but also enjoys a 95\% improvement in covariance propagation (and reduction) accuracy.

\begin{figure}[!htb]
    \centering
    \includegraphics[width=\linewidth]{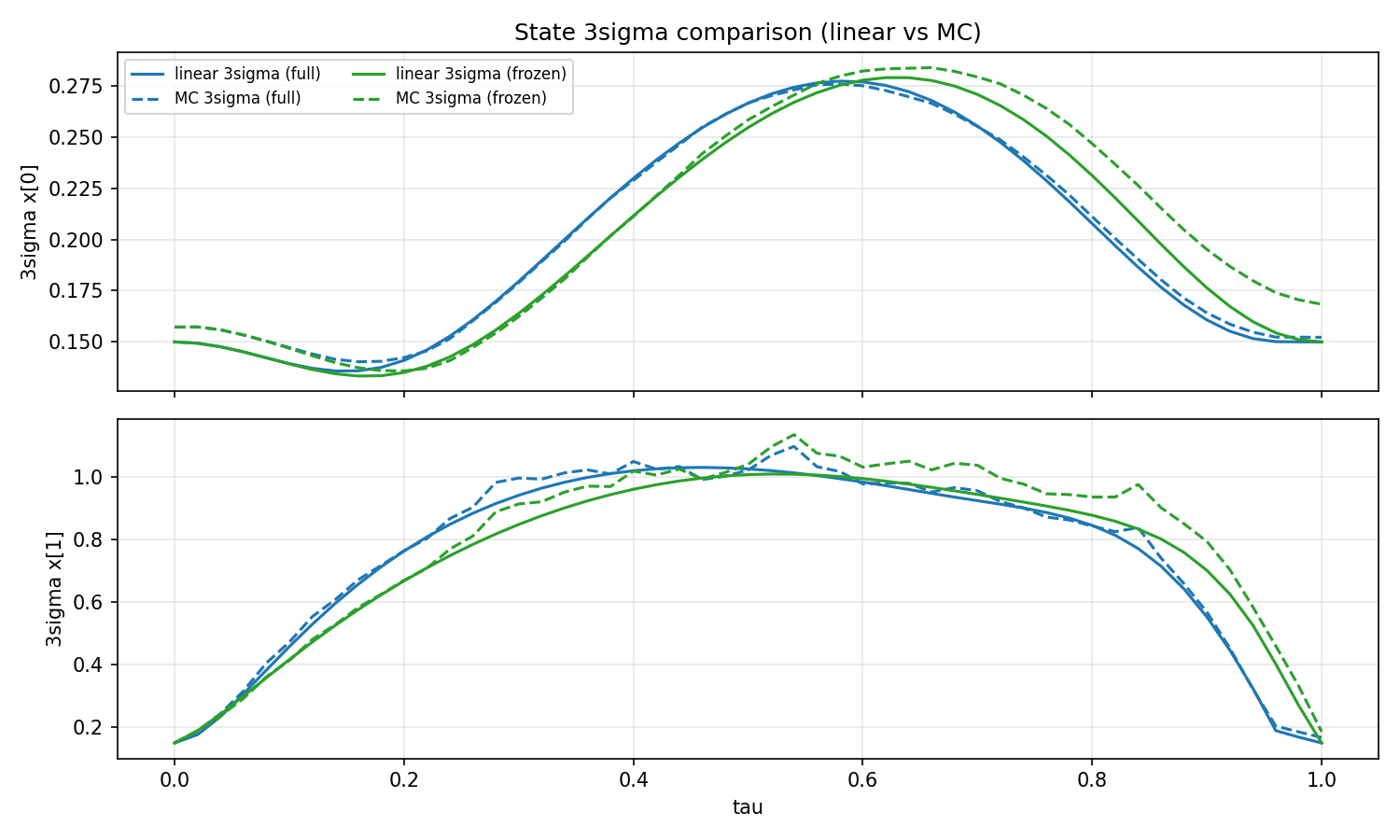}
    \caption{State $3\sigma$ standard deviation under (solid) linear covariance propagation, and (dashed) empirical nonlinear SDE Monte Carlo.}
    \label{fig:frozen-full-3sigma}
\end{figure}

\section{Conclusion}~\label{sec:conclusion}
In this work, we presented a free-final-time iterative covariance-steering (\texttt{FFT-iCS}) method that jointly optimizes interval-wise time allocation and an affine deviation-feedback policy for nonlinear SDEs with additive and multiplicative diffusion. The approach is built on a locally valid discrete-time model that preserves the first-order diffusion linearization under time dilation, leading to tractable SDP subproblems with terminal moment constraints and distributionally robust chance constraints. Future work will target higher-dimensional guidance applications and extensions to stopping-time/first-reach specifications.

%%%%%%%%%%%%%%%%%%%%%%%%%%%%%%%%%%%%%%%%%%%%%%%%%%%%%%%%%%%%%%%%%%%%%%%%%%%%%%%%
\appendix
\section{APPENDIX}

\subsection{Continuous-time linearization}~\label{app:ct-lin}
The matrices and vectors involved in the linearization of \eqref{eq:nl-sde-scaled} are computed as follows.
For the drift term $\sigma f(x,u)$,
\begin{align*}
    A_{\tau} &= \hat{\sigma}_{\tau}\frac{\partial f}{\partial x}\bigg|_{\hat{z}_{\tau}}, \qquad
    B_{\tau} = \hat{\sigma}_{\tau}\frac{\partial f}{\partial u}\bigg|_{\hat{z}_{\tau}}, \qquad
    c_{\tau} = f(\hat{x}_{\tau},\hat{u}_{\tau}),\\
    d_{\tau} &= \hat{\sigma}_\tau f(\hat{x}_\tau,\hat{u}_\tau) - A_{\tau} \hat{x}_{\tau} - B_{\tau} \hat{u}_{\tau} - c_{\tau}\hat{\sigma}_\tau \\ 
    &= -A_{\tau}\hat{x}_{\tau} - B_{\tau}\hat{u}_{\tau}.
\end{align*}
For each diffusion channel $\sqrt{\sigma}\,g_i(x,u)$,
\begin{align*}
    \tilde{A}_{\tau}^{(i)} &= \sqrt{\hat{\sigma}_{\tau}}\frac{\partial g_i}{\partial x}\bigg|_{\hat{z}_\tau}\qquad
    \tilde{B}_{\tau}^{(i)} = \sqrt{\hat{\sigma}_{\tau}}\frac{\partial g_i}{\partial u}\bigg|_{\hat{z}_{\tau}}\\ 
    \tilde{c}_{\tau}^{(i)} &= \frac{1}{2\sqrt{\hat{\sigma}_\tau}} g_i(\hat{x}_{\tau},\hat{u}_{\tau}),\\
    \tilde{d}_{\tau}^{(i)} &= \sqrt{\hat{\sigma}_{\tau}}g_i(\hat{x}_\tau,\hat{u}_\tau)-\tilde{A}_{\tau}^{(i)} \hat{x}_{\tau} - \tilde{B}_{\tau}^{(i)} \hat{u}_{\tau} - \tilde{c}_{\tau}^{(i)}\hat{\sigma}_\tau \\
    &=\frac{1}{2}\sqrt{\hat{\sigma}_\tau} g_i(\hat{x}_\tau,\hat{u}_\tau) - \tilde{A}_{\tau}^{(i)} \hat{x}_{\tau} - \tilde{B}_{\tau}^{(i)} \hat{u}_{\tau}.
\end{align*}

%%%%%%%%%%%%%%%%%%%%%%%%%%%%%%%%%%%%%%%%%%%%%%%%%%%%%%%%%%%%%%%%%%%%%%%%%%%%%%%%

\subsection{Exact mild solution and discretization terms}~\label{app:exact-disc}
As mentioned in the main text, the (exact) mild solution of \eqref{eq:linear-sde-aug} on $[\tau_k,\tau_{k+1}]$ satisfies \eqref{eq:exact-discrete}, where the deterministic discretization terms are
\begin{subequations}\label{eq:disc-terms-drift}
\begin{align}
    A_k &\coloneq \Phi(\tau_{k+1},\tau_k), \\ 
    F_k &\coloneq \int_{\tau_k}^{\tau_{k+1}}\Phi(\tau_{k+1},\tau)F_\tau\,\d\tau, \\ 
    d_k &\coloneq \int_{\tau_k}^{\tau_{k+1}}\Phi(\tau_{k+1},\tau)d_\tau\,\d\tau.
\end{align}
\end{subequations}
In addition, for the projected diffusion discretization used in \eqref{eq:Hbar-affine-cols}, the averaged (projected) matrices/vectors are given as
\begin{subequations}~\label{eq:disc-terms-averaged}
\begin{align}
    \tilde A^{(i)}_k
    &:=
    \frac{1}{\Delta\tau_k}\int_{\tau_k}^{\tau_{k+1}}\Phi(\tau_{k+1},\tau)\tilde A^{(i)}_\tau\,\d\tau,\label{eq:disc-terms-diffA}\\
    \tilde F^{(i)}_k
    &:=
    \frac{1}{\Delta\tau_k}\int_{\tau_k}^{\tau_{k+1}}\Phi(\tau_{k+1},\tau)\tilde F^{(i)}_\tau\,\d\tau,\label{eq:disc-terms-diffF}\\
    \tilde d^{(i)}_k
    &:=
    \frac{1}{\Delta\tau_k}\int_{\tau_k}^{\tau_{k+1}}\Phi(\tau_{k+1},\tau)\tilde d^{(i)}_\tau\,\d\tau.\label{eq:disc-terms-diffd}
\end{align}
\end{subequations}

The integrals in \eqref{eq:disc-terms-drift}--\eqref{eq:disc-terms-averaged} can be computed without numerical quadrature by integrating a single augmented linear ODE forward on each interval.
Define, for $\tau\in[\tau_k,\tau_{k+1}]$, the matrix-valued functions
\begin{subequations}~\label{eq:S-def1}
\begin{align}
    S_F(\tau) &:= \int_{\tau_k}^{\tau}\Phi(\tau,s)F_s\,\d s \in \R^{n\times(m+1)},\label{eq:SF-def}\\
    s_d(\tau) &:= \int_{\tau_k}^{\tau}\Phi(\tau,s)d_s\,\d s \in \R^{n},\label{eq:sd-def}
\end{align}
\end{subequations}
and for each diffusion channel $i\in\{1,\ldots,d\}$,
\begin{subequations}~\label{eq:S-def2}
\begin{align}
    S_{\tilde A}^{(i)}(\tau) &:= \int_{\tau_k}^{\tau}\Phi(\tau,s)\tilde A^{(i)}_s\,\d s \in \R^{n\times n},\label{eq:SAtil-def}\\
    S_{\tilde F}^{(i)}(\tau) &:= \int_{\tau_k}^{\tau}\Phi(\tau,s)\tilde F^{(i)}_s\,\d s \in \R^{n\times(m+1)},\label{eq:SFtil-def}\\
    s_{\tilde d}^{(i)}(\tau) &:= \int_{\tau_k}^{\tau}\Phi(\tau,s)\tilde d^{(i)}_s\,\d s \in \R^{n}.\label{eq:sdtil-def}
\end{align}
\end{subequations}
Differentiating the definitions in \eqref{eq:S-def1}--\eqref{eq:S-def2} and using \eqref{eq:stm-def} yields the coupled ODEs
\begin{subequations}
\begin{align}
    \dot \Phi(\tau) &= A_\tau \Phi(\tau), & \Phi(\tau_k)&=I,\label{eq:aug-phi}\\
    \dot S_F(\tau) &= A_\tau S_F(\tau)+F_\tau, & S_F(\tau_k)&=0,\label{eq:aug-SF}\\
    \dot s_d(\tau) &= A_\tau s_d(\tau)+d_\tau, & s_d(\tau_k)&=0,\label{eq:aug-sd}\\
    \dot S_{\tilde A}^{(i)}(\tau) &= A_\tau S_{\tilde A}^{(i)}(\tau)+\tilde A^{(i)}_\tau, & S_{\tilde A}^{(i)}(\tau_k)&=0,\label{eq:aug-SAt}\\
    \dot S_{\tilde F}^{(i)}(\tau) &= A_\tau S_{\tilde F}^{(i)}(\tau)+\tilde F^{(i)}_\tau, & S_{\tilde F}^{(i)}(\tau_k)&=0,\label{eq:aug-SFt}\\
    \dot s_{\tilde d}^{(i)}(\tau) &= A_\tau s_{\tilde d}^{(i)}(\tau)+\tilde d^{(i)}_\tau, & s_{\tilde d}^{(i)}(\tau_k)&=0.\label{eq:aug-sdt}
\end{align}
\end{subequations}
Evaluating at $\tau=\tau_{k+1}$ gives, directly,
\begin{equation}~\label{eq:disc-from-aug}
    A_k=\Phi(\tau_{k+1}),\quad
    F_k=S_F(\tau_{k+1}),\quad
    d_k=s_d(\tau_{k+1}),
\end{equation}
and similarly, for each $i$,
\begin{equation}~\label{eq:disc-from-aug-diff}
\begin{aligned}
    \tilde A^{(i)}_k&=\frac{1}{\Delta\tau_k}S_{\tilde A}^{(i)}(\tau_{k+1}),\quad
    \tilde F^{(i)}_k=\frac{1}{\Delta\tau_k}S_{\tilde F}^{(i)}(\tau_{k+1}), \\
    \tilde d^{(i)}_k&=\frac{1}{\Delta\tau_k}s_{\tilde d}^{(i)}(\tau_{k+1}). 
\end{aligned}
\end{equation}

For efficient implementation, it is convenient to stack all unknowns into one augmented matrix so that \eqref{eq:aug-phi}--\eqref{eq:aug-sdt} are solved in one ODE call.
Let
\begin{align}
M(\tau)
&:=
\Big[
\ \Phi(\tau)\ \ \ S_F(\tau)\ \ \ s_d(\tau)\ \ \ S_{\tilde A}^{(1)}(\tau)\ \cdots\ S_{\tilde A}^{(d)}(\tau) \nonumber\\
&S_{\tilde F}^{(1)}(\tau)\ \cdots\ S_{\tilde F}^{(d)}(\tau)\ \ \ s_{\tilde d}^{(1)}(\tau)\ \cdots\ s_{\tilde d}^{(d)}(\tau)\
\Big], \label{eq:stacked-M}
\end{align}
where we interpret each vector $s(\tau)\in\R^n$ as an $n\times 1$ block.
Then $M(\tau)$ satisfies the compact linear ODE
\begin{equation}\label{eq:M-ode}
    \dot M(\tau)=A_\tau M(\tau)+R(\tau),\quad M(\tau_k)=\big[I\ \ 0\ \ 0\ \cdots\ 0\big],
\end{equation}
with a right-hand side matrix $R(\tau)$ formed by horizontally stacking the corresponding inhomogeneities
\begin{align*}
    R(\tau)=\big[\ 0\ \ F_\tau\ \ d_\tau\ \ &\tilde A_\tau^{(1)}\ \cdots\ \tilde A_\tau^{(d)}\\ 
    &\hspace{15pt}\tilde F_\tau^{(1)}\ \cdots\ \tilde F_\tau^{(d)}\ \ \tilde d_\tau^{(1)}\ \cdots\ \tilde d_\tau^{(d)}\ \big].
\end{align*}
Thus, for each $k$, one integrates \eqref{eq:M-ode} over $[\tau_k,\tau_{k+1}]$ and then extracts the blocks at $\tau_{k+1}$ according to \eqref{eq:disc-from-aug}--\eqref{eq:disc-from-aug-diff}.
Given a reference trajectory $\hat z_\tau$ (hence known coefficient functions on each interval), the ODE \eqref{eq:M-ode} is \emph{independent across $k$}. Therefore the set $\{A_k,F_k,d_k,\tilde A_k^{(i)},\tilde F_k^{(i)},\tilde d_k^{(i)}\}_{k=0}^{N-1}$ can be computed in parallel over $k$ (and vectorized over the stacked blocks in $M$ within each interval).

%%%%%%%%%%%%%%%%%%%%%%%%%%%%%%%%%%%%%%%%%%%%%%%%%%%%%%%%%%%%%%%%%%%%%%%%%%%%%%%%

\subsection{Optimality of diffusion projection}~\label{app:proj-opt}

\begin{lem}[Orthogonal projection]~\label{lem:proj-opt}
    Fix the time step $k\in\{0,\ldots,N-1\}$ and condition on $\scriptF_{\tau_k}$.
    Among all constant matrices $G\in\R^{n\times d}$, the choice $G=\bar H_{k}$, with $\bar H_{k}$ given in \eqref{eq:Hbar}, uniquely minimizes the conditional mean-square error
    \begin{equation}\label{eq:proj-mse-opt}
        \E_k\!\left[\left\|\int_{\tau_k}^{\tau_{k+1}} H_{k}(\tau)\,\d\vct w_\tau - G\,\Delta\vct w_{k}\right\|^2\right].
    \end{equation}
\end{lem}
\begin{proof}
Using $\Delta \vct w_k=\vct w_{\tau_{k+1}}-\vct w_{\tau_k}
=\int_{\tau_k}^{\tau_{k+1}} I_d\,\d \vct w_\tau$, it follows that
\[
G\,\Delta\vct w_k=\int_{\tau_k}^{\tau_{k+1}} G\,\d \vct w_\tau,
\]
and therefore
\[
\int_{\tau_k}^{\tau_{k+1}} H_k(\tau)\,\d\vct w_\tau - G\,\Delta\vct w_k
=
\int_{\tau_k}^{\tau_{k+1}} (H_k(\tau)-G)\,\d\vct w_\tau.
\]
Conditioned on $\scriptF_{\tau_k}$, the integrand is square-integrable and deterministic (under the local one-step approximation in \eqref{eq:freeze-x}), so the conditional It\^{o} isometry gives
\begin{align}
J(G)
&:=\E_k\!\left[\left\|\int_{\tau_k}^{\tau_{k+1}} H_k(\tau)\,\d\vct w_\tau - G\,\Delta\vct w_k\right\|^2\right] \nonumber\\
&= \E_k\!\left[\left\|\int_{\tau_k}^{\tau_{k+1}} (H_k(\tau)-G)\,\d\vct w_\tau\right\|^2\right] \nonumber\\
&= \int_{\tau_k}^{\tau_{k+1}}\|H_k(\tau)-G\|_F^2\,\d\tau. \label{eq:proj-proof-JG}
\end{align}
Hence the stochastic minimization reduces to the deterministic least-squares problem
\[
\min_{G\in\R^{n\times d}} \int_{\tau_k}^{\tau_{k+1}}\|H_k(\tau)-G\|_F^2\,\d\tau.
\]
Let $\Delta\tau_k:=\tau_{k+1}-\tau_k$ and define
\[
\bar H_k := \frac{1}{\Delta\tau_k}\int_{\tau_k}^{\tau_{k+1}} H_k(\tau)\,\d\tau.
\]
Expanding the square and differentiating with respect to $G$ yields
\[
\nabla_G J(G)
=
-2\int_{\tau_k}^{\tau_{k+1}} H_k(\tau)\,\d\tau + 2\Delta\tau_k\,G.
\]
Setting $\nabla_G J(G)=0$ gives the unique minimizer
\[
G^\star = \frac{1}{\Delta\tau_k}\int_{\tau_k}^{\tau_{k+1}} H_k(\tau)\,\d\tau = \bar H_k,
\]
where uniqueness follows from strict convexity of $J$ in $G$ (equivalently, its Hessian is $2\Delta\tau_k I \succ 0$ in vectorized coordinates).
\end{proof}

%%%%%%%%%%%%%%%%%%%%%%%%%%%%%%%%%%%%%%%%%%%%%%%%%%%%%%%%%%%%%%%%%%%%%%%%%%%%%%%%

\subsection{Assumptions and auxiliary lemmas for Theorem~\ref{thm:one-step-error}}~\label{app:assums-and-proofs}

We state a local regularity assumption directly on the \emph{time-normalized nonlinear SDE coefficients}, from which the bounds used in the one-step error analysis follow.

\begin{assum}[Local smoothness]\label{assum:local-C1}
Fix a time step $k\in\{0,\ldots,N-1\}$ and interval $\mathcal{T}_k:=[\tau_k,\tau_{k+1}]$.
Consider the nonlinear SDE \eqref{eq:nl-sde-scaled}, where
$f:\mathcal{T}_k\times\R^n\times\R^m\to\R^n$ and
$g^{(i)}:\mathcal{T}_k\times\R^n\times\R^m\to\R^n$ for $i=1,\ldots,d$.
Assume the following:
\begin{enumerate}
    \item \textbf{Local smoothness.}
    The map $f$ is $C^1$ in $(\vct x,\vct u)$ and continuous in $\tau$, and each $g^{(i)}$ is $C^{1,1}$ in $(\vct x,\vct u)$ and continuous in $\tau$, on a neighborhood of a compact set
    $\mathcal K_k\subset \R^n\times\R^m$.

    \item \textbf{Bounded reference tube.}
    The current SCP reference pair $(\hat{\vct x}_\tau,\hat{\vct u}_\tau)$ satisfies
    \[
    (\hat{\vct x}_\tau,\hat{\vct u}_\tau)\in \mathcal K_k,
    \qquad \forall \tau\in \mathcal{T}_k,
    \]
    and $\tau\mapsto (\hat{\vct x}_\tau,\hat{\vct u}_\tau)$ is Lipschitz on $\mathcal{T}_k$.
\end{enumerate}
\end{assum}

\begin{rem}\label{rem:local-boundedness}
Assumption~\ref{assum:local-C1} is a \emph{local one-step} regularity condition.
In SCP implementations, the iterate is kept in a bounded neighborhood of the current reference by trust-region constraints, so the compact set $\mathcal K_k$ is natural.
By continuity on the compact set $\mathcal T_k\times\mathcal K_k$, all coefficient values and Jacobians needed for the local linearization are uniformly bounded on $\mathcal T_k$. Consequently, the associated state transition matrix (STM) is uniformly bounded on $\mathcal T_k$, and for fixed $(\vct x_k,\tilde{\vct u}_k)$ the frozen integrand $\tau\mapsto H_k(\tau;\vct x_k,\tilde{\vct u}_k)$ is Lipschitz on $\mathcal T_k$.
\end{rem}

\begin{lem}[Local conditional moment bounds]\label{lem:local-moment-increment}
Under Assumption~\ref{assum:local-C1}, there exist deterministic constants
$C_k^{(2)},\,C_k^{(\Delta)}<\infty$ such that the following hold for the local linearized SDE under ZOH for all $\tau\in\mathcal T_k$:
\begin{enumerate}
    \item[(i)]
    \begin{equation}\label{eq:second-moment-bound}
    \sup_{s\in[\tau_k,\tau]}\E_k\|\vct x_s\|^2
    \le
    C_k^{(2)}\bigl(1+\|\vct x_k\|^2+\|\tilde{\vct u}_k\|^2\bigr)
    \quad \text{a.s.}
    \end{equation}

    \item[(ii)]
    \begin{equation}\label{eq:ms-increment}
        \E_k\!\left[\|\vct x_\tau-\vct x_k\|^2\right]
        \le C_k^{(\Delta)}\,(\tau-\tau_k)\bigl(1+\|\vct x_k\|^2+\|\tilde{\vct u}_k\|^2\bigr).
    \end{equation}
\end{enumerate}
\end{lem}

\begin{proof}
For $s\in[\tau_k,\tau]$, write the local linearized SDE under ZOH as
\[
\d\vct x_s=b_s\,\d s+\Sigma_s\,\d\vct w_s,
\]
with
\[
b_s:=A_s\vct x_s+F_s\tilde{\vct u}_k+d_s,
\]
and
\begin{align*}
\Sigma_s=
\big[
\tilde A_s^{(1)}\vct x_s+\tilde F_s^{(1)}&\tilde{\vct u}_k+\tilde d_s^{(1)},
\ldots \\ 
&\tilde A_s^{(d)}\vct x_s+\tilde F_s^{(d)}\tilde{\vct u}_k+\tilde d_s^{(d)}
\big]\in\R^{n\times d}.
\end{align*}
By Assumption~\ref{assum:local-C1} and Remark~\ref{rem:local-boundedness}, all coefficient matrices/vectors
$A_s,F_s,d_s,\tilde A_s^{(i)},\tilde F_s^{(i)},\tilde d_s^{(i)}$
are uniformly bounded on $\mathcal T_k$.
Hence there exists a deterministic constant $C>0$ such that, for all $s\in\mathcal T_k$,
\begin{equation}\label{eq:bs-Sigmas-linear-growth}
\|b_s\|^2+\|\Sigma_s\|_\mathrm{F}^2
\le
C\left(1+\|\vct x_s\|^2+\|\tilde{\vct u}_k\|^2\right)
\qquad \text{a.s.}
\end{equation}

We first prove \eqref{eq:second-moment-bound}. Apply It\^{o}'s formula to $\phi(x)=\|x\|^2=x^\top x$. Since $\nabla\phi(x)=2x$ and $\nabla^2\phi(x)=2I$,
\[
\d\|\vct x_s\|^2
=
\left(2\vct x_s^\top b_s+\|\Sigma_s\|_{\mathrm{F}}^2\right)\d s
+2\vct x_s^\top\Sigma_s\,\d\vct w_s.
\]
Integrating from $\tau_k$ to $s$ yields
\[
\|\vct x_s\|^2
=
\|\vct x_k\|^2
+\int_{\tau_k}^{s}\left(2\vct x_r^\top b_r+\|\Sigma_r\|_{\mathrm{F}}^2\right)\d r
+\int_{\tau_k}^{s}2\vct x_r^\top\Sigma_r\,\d\vct w_r.
\]
Now take the conditional expectation given $\scriptF_{\tau_k}$. Since $\vct x_k$ is $\scriptF_{\tau_k}$-measurable and the It\^{o} integral has zero conditional mean,
\[
\E_k\|\vct x_s\|^2
=
\|\vct x_k\|^2
+\int_{\tau_k}^{s}\E_k\!\left[2\vct x_r^\top b_r+\|\Sigma_r\|_{\mathrm{F}}^2\right]\d r.
\]
Using $2x^\top b\le \|x\|^2+\|b\|^2$ gives
\[
\E_k\|\vct x_s\|^2
\le
\|\vct x_k\|^2
+\int_{\tau_k}^{s}\E_k\!\left[\|\vct x_r\|^2+\|b_r\|^2+\|\Sigma_r\|_{\mathrm{F}}^2\right]\d r.
\]
Applying \eqref{eq:bs-Sigmas-linear-growth} and absorbing constants into $C$,
\[
\E_k\|\vct x_s\|^2
\le
\|\vct x_k\|^2
+ C\int_{\tau_k}^{s}\left(1+\|\tilde{\vct u}_k\|^2\right)\d r
+ C\int_{\tau_k}^{s}\E_k\|\vct x_r\|^2\,\d r.
\]
Define $y(s):=\E_k\|\vct x_s\|^2$. Since $s-\tau_k\le \Delta\tau_k$,
\[
y(s)\le A_k + C\int_{\tau_k}^{s}y(r)\,\d r,
\,
A_k:=\|\vct x_k\|^2+C\Delta\tau_k\left(1+\|\tilde{\vct u}_k\|^2\right).
\]
By Gronwall's inequality,
\[
y(s)\le A_k e^{C(s-\tau_k)}\le A_k e^{C\Delta\tau_k},\qquad s\in[\tau_k,\tau].
\]
Absorbing deterministic factors depending only on local regularity data and $\Delta\tau_k$ into a constant $C_k^{(2)}$ gives \eqref{eq:second-moment-bound}.

We now prove \eqref{eq:ms-increment}. Write
\[
\vct x_\tau-\vct x_k
=
\int_{\tau_k}^{\tau} b_s\,\d s + \int_{\tau_k}^{\tau}\Sigma_s\,\d\vct w_s.
\]
Using a similar procedure to that utilized above, we obtain the bound
\begin{align*}
\E_k\|\vct x_\tau-\vct x_k\|^2
&\le
2(\tau-\tau_k)\int_{\tau_k}^{\tau}\E_k\|b_s\|^2\,\d s \\ 
&\hspace{2.5cm}+2\int_{\tau_k}^{\tau}\E_k\|\Sigma_s\|_{\mathrm{F}}^2\,\d s.
\end{align*}
Applying \eqref{eq:bs-Sigmas-linear-growth} and then \eqref{eq:second-moment-bound} yields
\[
\E_k\|\vct x_\tau-\vct x_k\|^2
\le C_k^{(\Delta)}\,(\tau-\tau_k)\left(1+\|\vct x_k\|^2+\|\tilde{\vct u}_k\|^2\right),
\]
after absorbing deterministic constants into $C_k^{(\Delta)}$.
\end{proof}

\begin{lem}[Lipschitz variance bound]\label{lem:lipschitz-variance}
Let $h>0$ and let $\psi:[0,h]\to \R^{p}$ be Lipschitz with constant $L$.
Let
\[
\bar \psi:=\frac{1}{h}\int_0^h \psi(t)\,\d t.
\]
Then
\begin{equation}\label{eq:lipschitz-variance}
\int_0^h \|\psi(t)-\bar \psi\|^2\,\d t \le \frac{L^2}{12}h^3.
\end{equation}
\end{lem}

\begin{proof}
We first derive the identity
\begin{equation}\label{eq:variance-pairwise-identity}
\int_0^h \|\psi(t)-\bar\psi\|^2\,\d t
=
\frac{1}{2h}\int_0^h\int_0^h \|\psi(t)-\psi(s)\|^2\,\d s\,\d t.
\end{equation}
Indeed, expanding the square and using the definition of $\bar\psi$,
\begin{align*}
\int_0^h &\|\psi(t)-\bar\psi\|^2\,\d t \\ 
&=
\int_0^h \|\psi(t)\|^2\,\d t
- 2\int_0^h \psi(t)^\top \bar\psi\,\d t
+ \int_0^h \|\bar\psi\|^2\,\d t \\
&=
\int_0^h \|\psi(t)\|^2\,\d t - h\|\bar\psi\|^2.
\end{align*}
On the other hand,
\begin{align*}
\int_0^h&\int_0^h \|\psi(t)-\psi(s)\|^2\,\d s\,\d t \\ 
&=
\int_0^h\int_0^h \Big(\|\psi(t)\|^2+\|\psi(s)\|^2-2\psi(t)^\top\psi(s)\Big)\,\d s\,\d t \\
&=
2h\int_0^h \|\psi(t)\|^2\,\d t
-2\left\|\int_0^h \psi(t)\,\d t\right\|^2 \\
&=
2h\int_0^h \|\psi(t)\|^2\,\d t - 2h^2\|\bar\psi\|^2.
\end{align*}
Dividing by $2h$ yields \eqref{eq:variance-pairwise-identity}.
Now apply the Lipschitz bound $\|\psi(t)-\psi(s)\|^2 \le L^2|t-s|^2$.
Using \eqref{eq:variance-pairwise-identity}, it follows that
\begin{align*}
\int_0^h \|\psi(t)-\bar\psi\|^2\,\d t
&\le
\frac{L^2}{2h}\int_0^h\int_0^h |t-s|^2\,\d s\,\d t \\
&=
\frac{L^2}{2h}\cdot \frac{h^4}{6}
=
\frac{L^2}{12}h^3,
\end{align*}
which proves \eqref{eq:lipschitz-variance}.
\end{proof}

%%%%%%%%%%%%%%%%%%%%%%%%%%%%%%%%%%%%%%%%%%%%%%%%%%%%%%%%%%%%%%%%%%%%%%%%%%%%%%%%

\subsection{Proof of Theorem~\ref{thm:one-step-error}}\label{app:proof-thm1}
\begin{proof}
\textit{Step 1 (freezing error).}
Subtracting the frozen update from the exact mild update \eqref{eq:exact-discrete} gives
\begin{align*}
\vct e_{k+1}^{(x)}
=
\sum_{i=1}^d \int_{\tau_k}^{\tau_{k+1}}
\Phi(\tau_{k+1},\tau)\tilde A_\tau^{(i)}(\vct x_\tau-\vct x_k)\,\d \vct w_\tau^{(i)}.
\end{align*}
Define
\[
M_{k,i}(\tau):=\Phi(\tau_{k+1},\tau)\tilde A_\tau^{(i)}.
\]
By conditional It\^{o} isometry,
\begin{align*}
\E_k\|\vct e_{k+1}^{(x)}\|^2
&=
\sum_{i=1}^d \int_{\tau_k}^{\tau_{k+1}}
\E_k\!\left[\|M_{k,i}(\tau)(\vct x_\tau-\vct x_k)\|^2\right]\d\tau \\
&\le
\sum_{i=1}^d \int_{\tau_k}^{\tau_{k+1}}
\|M_{k,i}(\tau)\|^2\,\E_k\|\vct x_\tau-\vct x_k\|^2\,\d\tau.
\end{align*}
By Assumption~\ref{assum:local-C1} and Remark~\ref{rem:local-boundedness}, there exists a deterministic constant $C_{M,k}<\infty$ such that
\[
\|M_{k,i}(\tau)\|\le C_{M,k},\qquad \forall \tau\in\mathcal T_k,\ \forall i=1,\ldots,d.
\]
Applying Lemma~\ref{lem:local-moment-increment}(ii) and using
\[
\int_{\tau_k}^{\tau_{k+1}}(\tau-\tau_k)\,\d\tau=\frac{1}{2}\Delta\tau_k^2
\]
yields
\[
\E_k\|\vct e_{k+1}^{(x)}\|^2
\le
C_k^{(x)}\Delta\tau_k^2\left(1+\|\vct x_k\|^2+\|\tilde{\vct u}_k\|^2\right),
\]
which proves \eqref{eq:freeze-bound}.

\textit{Step 2 (projection error).}
By definition of the projected update,
\[
\vct e_{k+1}^{(p)}
=
\int_{\tau_k}^{\tau_{k+1}} \big(H_k(\tau)-\bar H_k\big)\,\d\vct w_\tau.
\]
Conditional It\^{o} isometry (c.f. Appendix~\ref{app:proj-opt}) gives
\[
\E_k\!\left[\|\vct e_{k+1}^{(p)}\|^2\right]
=
\int_{\tau_k}^{\tau_{k+1}}\|H_k(\tau)-\bar H_k\|_F^2\,\d\tau.
\]
By Assumption~\ref{assum:local-C1} and Remark~\ref{rem:local-boundedness}, the map
$\tau\mapsto H_k(\tau;\vct x_k,\tilde{\vct u}_k)$ is Lipschitz on $\mathcal T_k$ for fixed $(\vct x_k,\tilde{\vct u}_k)$.
Let $L_{H,k}(\vct x_k,\tilde{\vct u}_k)$ denote a corresponding Lipschitz constant.
Applying Lemma~\ref{lem:lipschitz-variance} on an interval of length $\Delta\tau_k$ yields
\[
\E_k\!\left[\|\vct e_{k+1}^{(p)}\|^2\right]
\le \frac{1}{12}L_{H,k}(\vct x_k,\tilde{\vct u}_k)^2\Delta\tau_k^3,
\]
which proves \eqref{eq:proj-bound}.
\end{proof}

%%%%%%%%%%%%%%%%%%%%%%%%%%%%%%%%%%%%%%%%%%%%%%%%%%%%%%%%%%%%%%%%%%%%%%%%%%%%%%%%
\subsection{Proof of Theorem~\ref{thm:moment-prop}}\label{app:proof-moment-prop}
\begin{proof}
Let $\delta\vct x_k \coloneqq \vct x_k-\bar{\vct x}_k$ and $\delta\tilde{\vct u}_k\coloneqq \tilde{\vct u}_k-\E[\tilde{\vct u}_k]$.
From \eqref{eq:approx-discrete-general} and \eqref{eq:Hbar-affine-cols},
\begin{align*}
\vct x_{k+1}
&= A_k\vct x_k + F_k\tilde{\vct u}_k + \vct d_k \nonumber \\
&\hspace{1.5cm}+ \sum_{i=1}^d\left(\tilde A_k^{(i)}\vct x_k+\tilde F_k^{(i)}\tilde{\vct u}_k+\tilde d_k^{(i)}\right)\Delta\vct w_k^{(i)}.
\end{align*}
Taking expectations and using $\E[\Delta\vct w_k^{(i)}]=0$ yields \eqref{eq:mean-prop}.

For the covariance, write $\vct x_k=\bar{\vct x}_k+\delta\vct x_k$ and
$\tilde{\vct u}_k=\E[\tilde{\vct u}_k]+\delta\tilde{\vct u}_k$ and collect terms to obtain the deviation dynamics
\begin{align*}
\delta\vct x_{k+1}
&= A_k\delta\vct x_k + F_k\delta\tilde{\vct u}_k
+ \sum_{i=1}^d q_k^{(i)}\Delta\vct w_k^{(i)} \\
&\hspace{1.5cm}+ \sum_{i=1}^d\left(\tilde A_k^{(i)}\delta\vct x_k+\tilde F_k^{(i)}\delta\tilde{\vct u}_k\right)\Delta\vct w_k^{(i)},
\end{align*}
with $q_k^{(i)}$ as in Theorem~\ref{thm:moment-prop}.
Using independence of $\Delta\vct w_k$ and $\scriptF_{\tau_k}$, $\E[\Delta\vct w_k^{(i)}]=0$, and $\E[(\Delta\vct w_k^{(i)})^2]=\Delta\tau_k$, all cross-terms with odd powers of $\Delta\vct w_k^{(i)}$ vanish, and mixed-channel terms vanish by independence.
A straightforward expansion yields the desired result \eqref{eq:cov-prop}.
\end{proof}

%%%%%%%%%%%%%%%%%%%%%%%%%%%%%%%%%%%%%%%%%%%%%%%%%%%%%%%%%%%%%%%%%%%%%%%%%%%%%%%%

\subsection{Losslessness of the LMI relaxations}~\label{app:lossless}

\begin{thm}~\label{thm:cov-losslessness}
Consider Problem~\ref{prob:penalty-CS-problem}.
Assume $\Sigma_{x_k}\succ 0$ for all $k=0,\ldots,N-1$ and that the problem is feasible.
If the regularization satisfies
\begin{equation}~\label{eq:lossless-sufficient}
    \frac{\partial \scriptJ_{\mathrm{reg}}}{\partial Y_k} \succ 0,\,
    \frac{\partial \scriptJ_{\mathrm{reg}}}{\partial \tilde\Sigma_{ik}} \succ 0,\, \forall k=0,\ldots,N-1,\ \forall i=1,\ldots,d,
\end{equation}
then there exists an optimal solution $z^\star$ of Problem~\ref{prob:penalty-CS-problem} that satisfies the relaxations with equality, i.e.,
\begin{equation}~\label{eq:lossless-equalities}
    Y_k^\star = U_k^\star(\Sigma_{x_k}^\star)^{-1}(U_k^\star)^\intercal,\qquad
    \tilde\Sigma_{ik}^\star = \tilde q_k^{(i)}(z^\star)\,\tilde q_k^{(i)}(z^\star)^\intercal,
\end{equation}
for all $k=0,\ldots,N-1$ and all $i=1,\ldots,d$. In particular, the LMI relaxations \eqref{eq:cov-final-relaxed-2}--\eqref{eq:cov-final-relaxed-3} are lossless under \eqref{eq:lossless-sufficient}.
\end{thm}

\begin{proof}
We prove losslessness by contradiction using a feasible-descent ($\varepsilon$-) argument.
The key observation is that the relaxation variables $(Y_k,\tilde\Sigma_{ik})$ enter Problem~\ref{prob:penalty-CS-problem} only through the semidefinite constraints \eqref{eq:cov-final-relaxed-2}--\eqref{eq:cov-final-relaxed-3} and through the objective via $\scriptJ_{\mathrm{reg}}$.

\paragraph{Step 1 (tightness of \eqref{eq:cov-final-relaxed-2})}
Fix $k\in\{0,\ldots,N-1\}$ and consider an optimal solution $z^\star$ of Problem~\ref{prob:penalty-CS-problem}.
Denote $U_k^\star$ and $\Sigma_{x_k}^\star\succ 0$ the corresponding values.
Since $\Sigma_{x_k}^\star\succ 0$, the LMI \eqref{eq:cov-final-relaxed-2} is equivalent (by Schur complement) to
\[
Y_k^\star \succeq U_k^\star(\Sigma_{x_k}^\star)^{-1}(U_k^\star)^\intercal.
\]
Define the slack
\[
\Delta_k \coloneqq Y_k^\star - U_k^\star(\Sigma_{x_k}^\star)^{-1}(U_k^\star)^\intercal \succeq 0.
\]
If $\Delta_k=0$ we are done. Suppose for contradiction that $\Delta_k\neq 0$.
For any $\varepsilon\in(0,1]$, define the perturbed variable
\[
Y_k(\varepsilon) \coloneqq Y_k^\star - \varepsilon \Delta_k.
\]
We verify that $Y_k(\varepsilon)$ preserves feasibility. Using the definition of $\Delta_k$,
\begin{align*}
Y_k(\varepsilon)
&= Y_k^\star - \varepsilon\!\left(Y_k^\star - U_k^\star(\Sigma_{x_k}^\star)^{-1}(U_k^\star)^\intercal\right) \\ 
&= (1-\varepsilon)Y_k^\star + \varepsilon\,U_k^\star(\Sigma_{x_k}^\star)^{-1}(U_k^\star)^\intercal.
\end{align*}
Since $Y_k^\star \succeq U_k^\star(\Sigma_{x_k}^\star)^{-1}(U_k^\star)^\intercal$ and $1-\varepsilon\ge 0$, it follows that
\[
Y_k(\varepsilon)\succeq U_k^\star(\Sigma_{x_k}^\star)^{-1}(U_k^\star)^\intercal,
\]
hence \eqref{eq:cov-final-relaxed-2} remains satisfied with all other decision variables fixed.
Moreover, $Y_k(\varepsilon)\succeq 0$ because it is the sum of the PSD matrix $U_k^\star(\Sigma_{x_k}^\star)^{-1}(U_k^\star)^\intercal$ and the PSD matrix $(1-\varepsilon)\Delta_k$.
Thus, replacing $Y_k^\star$ by $Y_k(\varepsilon)$ yields a feasible point of Problem~\ref{prob:penalty-CS-problem} for every $\varepsilon\in(0,1]$.

Now consider the objective change. Since $\scriptJ_{\mathrm{reg}}$ is convex and differentiable in $Y_k$ and satisfies $\frac{\partial \scriptJ_{\mathrm{reg}}}{\partial Y_k}\succ 0$, the directional derivative of $\scriptJ_{\mathrm{reg}}$ at $Y_k^\star$ along $-\Delta_k$ is strictly negative:
\[
\left.\frac{\mathrm{d}}{\mathrm{d}\varepsilon}\scriptJ_{\mathrm{reg}}\big(Y_k^\star-\varepsilon\Delta_k,\tilde\Sigma^\star\big)\right|_{\varepsilon=0}
\hspace{-12pt}= -\left\langle \frac{\partial \scriptJ_{\mathrm{reg}}}{\partial Y_k}(Y_k^\star,\tilde\Sigma^\star),\Delta_k\right\rangle < 0.
\]
Therefore, for sufficiently small $\varepsilon>0$, replacing $Y_k^\star$ by $Y_k(\varepsilon)$ strictly decreases the objective while preserving feasibility, contradicting optimality of $z^\star$.
Hence $\Delta_k=0$, proving $Y_k^\star = U_k^\star(\Sigma_{x_k}^\star)^{-1}(U_k^\star)^\intercal$.

\paragraph{Step 2 (tightness of \eqref{eq:cov-final-relaxed-3})}
Fix $(i,k)$ and consider the corresponding optimal values in $z^\star$.
Since the scalar bottom-right entry in \eqref{eq:cov-final-relaxed-3} is $1>0$, the Schur complement yields
\[
\tilde\Sigma_{ik}^\star \succeq \tilde q_k^{(i)}(z^\star)\,\tilde q_k^{(i)}(z^\star)^\intercal.
\]
Define the slack
\[
\Delta_{ik}\coloneqq \tilde\Sigma_{ik}^\star - \tilde q_k^{(i)}(z^\star)\,\tilde q_k^{(i)}(z^\star)^\intercal \succeq 0.
\]
If $\Delta_{ik}\neq 0$, define $\tilde\Sigma_{ik}(\varepsilon)\coloneqq \tilde\Sigma_{ik}^\star-\varepsilon\Delta_{ik}$ for $\varepsilon\in(0,1]$.
As above, feasibility of \eqref{eq:cov-final-relaxed-3} is preserved for all such $\varepsilon$.
Since $\frac{\partial \scriptJ_{\mathrm{reg}}}{\partial \tilde\Sigma_{ik}}\succ 0$, the directional derivative of the regularizer along $-\Delta_{ik}$ is strictly negative, so for small $\varepsilon>0$ the objective strictly decreases while maintaining feasibility, contradicting optimality.
Therefore $\Delta_{ik}=0$ and $\tilde\Sigma_{ik}^\star = \tilde q_k^{(i)}(z^\star)\,\tilde q_k^{(i)}(z^\star)^\intercal$.

Applying Steps 1--2 for all indices yields \eqref{eq:lossless-equalities} and completes the proof.
\end{proof}
%%%%%%%%%%%%%%%%%%%%%%%%%%%%%%%%%%%%%%%%%%%%%%%%%%%%%%%%%%%%%%%%%%%%%%%%%%%%%%%%

\subsection{Numerical SDE Solution}~\label{app:sde-numerical-int}
In the numerical examples in Section~\ref{sec:examples}, it is necessary to compute the solution to a continuous-time SDE to demonstrate the efficacy of the proposed FFT-iCS solution.
To this end, we use Milstein's method to approximate the solution to the SDE \eqref{eq:nl-sde-scaled}, with both weak and strong order of convergence $O(\Delta\tau)$ \cite{KloedenPlaten1992NSDE}.
To this end, over a uniform partition $\scriptP = (\tau_0, \ldots, \tau_N)$ with interval width $\Delta\tau = 1/N$, the approximate solution is given by the following recursion.
For every Monte Carlo run $j = 1,\ldots, N_{\mathrm{mc}}$, we first sample the initial state as $x_0^{(j)}$ from $\scriptN(\mu_i,\Sigma_i)$.
The state at the end time step is then given by
\begin{align}
    x_{k+1}^{(j)} &= x_k^{(j)} + \sigma_k^\star f(x_k^{(j)}, u_k^{(j)}) \Delta\tau + \sqrt{\sigma_k^\star} g(x_k^{(j)}, u_k^{(j)}) \Delta w_k^{(j)} \nonumber\\ 
    &\hspace{-0.5cm}+ \sum_{i=1}^{d}\frac{\sigma_k^\star}{2} g_i(x_k^{(j)}, u_k^{(j)}) \frac{\partial g_i}{\partial x}(x_k^{(j)}, u_k^{(j)})\big((\Delta w_k^{(i,j)})^2 - \Delta\tau\big),~\label{eq:milton-prop}
\end{align}
where $u_k^{(j)} = v_k^\star + K_k^\star (x_k^{(j)} - \bar{x}_k)$ is the optimal control input with $\bar{x}_k$ from \eqref{eq:mean-prop-affine-control}, and $\Delta w_k^{(j)}$ is sampled from the Brownian increment $\mathcal{N}(0,\Delta\tau I_d)$.
%%%%%%%%%%%%%%%%%%%%%%%%%%%%%%%%%%%%%%%%%%%%%%%%%%%%%%%%%%%%%%%%%%%%%%%%%%%%%%%%

\bibliographystyle{IEEEtran}
\bibliography{refs.bib}

\end{document}